\documentclass[11pt,a4paper,reqno]{amsart}

\title[Tropical Fr\'echet Means]{Tropical Fr\'echet Means: \\ a polyhedral approach to exact optimization}

\author{Kamillo Ferry}
\address{Kamillo Ferry, Technical University Berlin, Germany}
\email[Corresponding author]{ferry@math.tu-berlin.de}

\author{Bo Lin}
\address{Bo Lin, Georgia Institute of Technology, USA}

\author{Carlos Am\'endola}
\address{Carlos Am\'endola, Technical University Berlin, Germany}

\author{Anthea Monod}
\address{Anthea Monod, Imperial College London, UK}

\author{Ruriko Yoshida}
\address{Ruriko Yoshida, Naval Postgraduate School, USA}


\newcommand{\xxxtodo}[1]{}

\usepackage{amsmath,amsfonts,amsthm,mathtools}
\usepackage{mathrsfs}

\usepackage{siunitx}

\usepackage{graphicx}
\usepackage[dvipsnames]{xcolor}

\usepackage{latexsym}
\usepackage{fancyvrb}
\usepackage{enumerate}

\usepackage{subcaption}

\usepackage{float}
\usepackage{bm}
\usepackage{bbm}
\usepackage{comment}
\usepackage{multicol,multirow}
\usepackage{array}
\usepackage{relsize}
\usepackage{chngcntr}
\usepackage{etoolbox}
\usepackage{caption}
\usepackage{tikz}
\usepackage{tikz-cd}
\usetikzlibrary{patterns}
\usepackage{amsopn}
\usepackage{calc}
\usepackage{algorithm}
\usepackage[noend]{algpseudocode}

\usepackage{pgfplots}
\pgfplotsset{compat=1.18} 
\usepackage{bchart}

\usepackage{standalone}

\usepackage{setspace}
\usepackage[headings]{fullpage}

\usepackage[backend=biber,style=numeric,maxbibnames=99,backref,backrefstyle=none,isbn=false]{biblatex}
\AtNextBibliography{\small}

\usepackage{url}
\usepackage[citecolor=blue, linkcolor=blue,urlcolor=blue]{hyperref}
\usepackage[nameinlink]{cleveref}

\newtheorem{theorem}{Theorem}
\newtheorem*{theorem*}{Theorem}
\newtheorem{corollary}[theorem]{Corollary}
\newtheorem{lemma}[theorem]{Lemma}
\newtheorem{proposition}[theorem]{Proposition}

\newtheorem*{claim*}{Claim}

\theoremstyle{definition}
\newtheorem{definition}[theorem]{Definition}
\newtheorem*{definition*}{Definition}

\theoremstyle{remark}
\newtheorem{remark}[theorem]{Remark}

\newtheorem*{example*}{Example}

\newtheorem{XxmpX}[theorem]{Example} 
\newenvironment{example}    
  {%
   \pushQED{\qed}\begin{XxmpX}}
  {\popQED\end{XxmpX}}

\def\beginmat{ \left( \begin{array} }
\def\endmat{ \end{array} \right) }

\makeatletter
\newcommand*{\op}{%
  \DOTSB
  \mathop{\vphantom{\bigoplus}\mathpalette\matt@op\relax}%
  \slimits@
}
\newcommand\matt@op[2]{%
  \vcenter{\m@th\hbox{\resizebox{\widthof{$#1\bigoplus$}}{!}{$\boxplus$}}}%
}
\makeatother

\newcommand{\one}{{\bf{1}}}

\def\R{{\mathbb R}}

\def\TT{{\mathbb T}}
\newcommand{\torus}[1][n]{\mathbb{R}^{#1}/\mathbb{R}\mathbf{1}}

\newcommand{\algO}{\ensuremath{\mathcal{O}}}
\newcommand{\bigO}{\algO}

\def\conv{{\mathrm{conv}}}
\def\tconv{{\mathrm{tconv}}}
\def\Q{{\mathrm{Q}}}

\newcommand{\Sym}{\mathbb{S}}
\newcommand{\dtr}{d_{\mathrm{tr}}}
\newcommand{\tfm}[1][x]{\ensuremath{\widetilde{#1}}}

\newcommand{\TropBall}[2]{\ensuremath{\mathrm{B_{tr}}(#1, #2)}}

\newcommand{\braidarr}[1][n]{\ensuremath{\mathcal{B}_{#1}}}

\usepackage{environ}
\usepackage{tabularx}
\makeatletter
\newcommand{\objectivefunction}[1]{\gdef\@objectivefunction{#1}}
\newcommand{\variables}[1]{\gdef\@variables{#1}}
\newcommand{\constraints}[1]{\gdef\@constraints{#1}}
\NewEnviron{optimproblem}{
  \objectivefunction{}\variables{}\constraints{}
  \BODY
  \par\addvspace{.5\baselineskip}
  \noindent
  \begin{tabularx}{\textwidth}{@{\hspace{\parindent}} l X p{30cm}}
      \textbf{Minimize} & \@objectivefunction \\
      \textbf{with respect to} & \@variables \\ 
      \textbf{such that} & \@constraints
  \end{tabularx}
  \par\addvspace{.5\baselineskip}
}
\makeatother

\makeatletter
\def\@biblabel#1{}
\makeatother

\begin{document}
\begin{abstract}

The Fr\'{e}chet mean is a fundamental notion of central tendency defined as a minimizer of a sum of squared distances in a general metric space.  In this paper, we study Fr\'{e}chet means in tropical geometry---a piecewise linear, combinatorial, and polyhedral variant of algebraic geometry---by formulating and solving the associated tropical quadratic optimization problem.  We give a geometric characterization of the collection of all tropical Fr\'{e}chet means as a bounded set that is simultaneously tropically and classically convex, hence a polytrope.  We establish the existence of positivity certificates for maxima of finitely many quadratic polynomials in $\mathbb{R}[x_1,\ldots,x_n]$ whose homogeneous quadratic components are sums of squares, which provides a symbolic framework for exact optimization.  Using this structure, we develop algorithms for computing tropical Fr\'{e}chet means and the associated Fr\'{e}chet mean polytrope.  We further describe a combinatorial type decomposition of the objective function induced by braid arrangements, yielding a piecewise quadratic representation and a fully symbolic method for exact computation.

\end{abstract}

\onehalfspace
\maketitle


\section{Introduction}
\label{sec:intro}

The Fr\'{e}chet mean is a classical notion of centrality defined as a minimizer of a sum of squared distances in a metric space.  Unlike in Euclidean spaces, where the mean admits a closed-form expression, Fréchet means in general metric spaces are defined implicitly as solutions to an optimization problem and they need not be unique.  This lack of smoothness and uniqueness makes both theoretical analysis and exact computation challenging.

In this paper, we study Fr\'{e}chet means in the tropical projective torus endowed with the tropical metric.  Tropical geometry replaces classical algebraic operations by max-plus arithmetic, which leads to a piecewise linear and polyhedral geometric setting.  Tropical formulations arise naturally in applications such as phylogenetics and algebraic statistics \cite{Maclagan.Sturmfels, MLYK:2022}; from a computational perspective they define a class of non-smooth optimization problems whose structure is fundamentally combinatorial.

A central difficulty is that the tropical Fr\'{e}chet objective function is a sum of squared tropical distances, each of which is itself a maximum of finitely many quadratic functions.  As a result, the global objective is a piecewise quadratic function defined by an exponential number of regions, and standard smooth optimization techniques are inapplicable.  Moreover, classical algorithms for Fr\'{e}chet means, such as Sturm’s algorithm \cite{Sturm:2003}, are known to fail in tropical settings due to the complex curvature properties of the tropical projective torus \cite{Amendola.Monod:2021}.

Our contributions are twofold: First, we provide a geometric characterization by showing that the set of all tropical Fr\'{e}chet means of a finite sample forms a bounded polytrope, i.e., a subset of tropical projective space that is both tropically and classically convex.  This implies that the set of Fr\'echet means admits both a tropical and a classical polyhedral description. We also show how once a single Fr\'{e}chet mean is known, one can compute all Fr\'{e}chet means. 

Our second contribution is exact symbolic computation of the tropical Fr\'{e}chet mean: we develop a symbolic framework for solving the tropical Fr\'{e}chet optimization problem without relying on numerical approximation.  More specifically, we first establish a general positivity certificate for maxima of quadratic polynomials with nonnegative homogeneous components.  This result guarantees the existence of exact symbolic lower bounds for the Fr\'{e}chet objective and provides a certificate-based method for verifying optimality.  Second, we introduce a type decomposition of the tropical Fr\'{e}chet objective function using braid arrangements.  We show that the objective function admits a finite piecewise quadratic structure indexed by the chambers of a braid arrangement associated with the input data.  Within each chamber, the objective reduces to an explicit sum of squares, and its critical locus is given by the solution of a linear system.  This leads to a purely symbolic algorithm that enumerates feasible covectors, computes affine critical loci, and intersects them with combinatorially defined regions in order to recover the full Fr\'{e}chet mean polytrope.  From the perspective of symbolic computation, this provides a concrete example of a non-smooth metric optimization problem that can be solved exactly using polyhedral and combinatorial techniques.  The resulting algorithms operate entirely in exact arithmetic and rely only on linear algebra, hyperplane arrangements, and convex polyhedral computations.

The paper is organized as follows.  We close this introduction with an overview of related work.  In Section 2, we review tropical metric geometry and tropical convexity.  In Section 3, we define tropical Fr\'{e}chet means and characterize their solution set as a polytrope.  Section 4 develops exact quadratic optimization methods and positivity certificates. Section 5 introduces the braid arrangement and type decomposition of the objective function and presents a symbolic algorithm for computing Fr\'{e}chet mean polytropes.  Section 6 provides numerical experiments.  We close the paper in Section 7 with a discussion on directions for future research.

\subsection{Related Work}

A similar statistical quantity of centrality that is defined by a similar optimization problem is the \emph{Fermat--Weber point}, which is a generalization of the median to general metric spaces.  These have been introduced in tropical settings identical to those of our work by \citeauthor{Lin.Yoshida:2018}~\cite{Lin.Yoshida:2018}  
 who also study their uniqueness properties.

\citeauthor{SBYM:2024}~\cite{SBYM:2024} studied the set of (symmetric) tropical Fermat-Weber points 
arriving at the conclusion that this set is a polytrope. This mirrors our results about
the set of tropical Fréchet means.

\citeauthor{MR4728892} give bounds on the dimension of the set of asymmetric
tropical Fermat--Weber points and propose an approach to its computation using
optimal transport \cite{MR4728892}. 
Again, they observed that the set of asymmetric tropical Fermat--Weber points
for a given sample forms a polytrope.

From an algorithmic perspective, tropical optimization problems have also been studied using descent-based methods.  In particular, \citeauthor{talbut2025tropical}~\cite{talbut2025tropical} introduced a general framework for \emph{tropical gradient descent}, providing a numerical optimization method for piecewise linear and piecewise smooth tropical objectives.  While their approach is designed for a wide and general class of statistical objectives on the tropical projective torus, this paper focuses on exact symbolic computation and exploits the full polyhedral and combinatorial structure of the tropical Fr\'{e}chet objective.

Both tropical Fr\'echet means and tropical Fermat--Weber points describe location problems of 
\emph{tropically quasiconvex functions} proposed and studied by \citeauthor{Comǎneci:2024}~\cite{Comǎneci:2024}, who further shows that this class of location problems enjoy particularly nice properties for
use as consensus methods in phylogenetics.


\section{Tropical Metric Spaces}

In this section, we present the space in which we work---the tropical projective torus---and the metric with which we endow this space---the tropical metric.  We also review key properties of this space that will be important for our study of tropical Fréchet means.

\subsection{The Tropical Projective Torus and the Tropical Metric}

Denote by \(\TT = \R\cup\{-\infty\}\) the (max-)tropical numbers with operations \(a\oplus b\coloneqq \max\{\, a,b \,\}\)
and \(a\odot b \coloneqq a+b\). For vectors $x, y \in \TT^n$ and $\lambda \in \R$, we define the tropical operations as follows:
$$x \oplus y = ( \max\{x_1,y_1\}, \dots, \max\{ x_n,y_n\}), \quad \lambda \odot x = (\lambda + x_1,\dots,\lambda + x_n).$$

Usual metrics in $\R^n$ do not behave well tropically, so we consider instead the following space to define an appropriate tropical distance. Define the equivalence relation in $\R^n$:
$$
x \sim y \Longleftrightarrow \mbox{~all coordinates of~} (x-y) \mbox{~are equal.} 
$$

The {\em tropical projective torus} $\R^n/\R\one$ is the quotient space given by the set of equivalence classes under $\sim$; 
in other words, it is the space constructed by identifying vectors that differ from each other by tropical scalar multiplication. 
Denote the equivalence class of $x$ by $[x] \in \R^{n}/\R\one$. 

The tropical projective torus is a metric space when endowed with the following distance function.

\begin{definition}
\label{def:tropicalmetric}
Let $[x], [y] \in \R^{n}/\R\one$. We define the \emph{tropical metric} on $\R^n/\R\one$ as
\begin{align*}
\dtr([x],[y]) :=& \max_{1\leq j < i \leq {n}}\big\lvert(x_i-x_j) - (y_i - y_j)\big|\\
 =& \max_{1 \leq i \leq n} (x_i - y_i) - \min_{1 \leq i \leq n} (x_i - y_i).
\end{align*}
\end{definition}

\begin{example}
    \(\dtr((4,0,9),(0,-1,5)) = \max \{4,1,4\}- \min \{4,1,4\} = 4-1 = 3.\)
\end{example}

\begin{proposition}[{\cite[Prop.\ 5.\,19]{Joswig:ETC}}]
The tropical metric is a well-defined metric.
\end{proposition}

The metric space on which we seek to study and solve the Fréchet mean optimization problem is then the tropical projective torus equipped with the tropical metric: $(\R^{n}/\R\one, \dtr)$.

\subsection{Tropical Convexity}

Convexity is a key geometric property that has implications on the definition and solution to Fréchet means.  It is a property that also makes sense in $(\R^{n}/\R\one, \dtr)$.  We now overview some known results and present some key concepts associated with tropical convexity that will be important in our work.

Recall that a \emph{(classical) polytope} \(P\subset\R^n\) is defined as the set of convex combinations
of a finite set of points \(\left\{\, a_1,\dots,a_m \,\right\}\subset\R^n\):
\[
    P = \conv(a_1,\dots,a_m) \coloneqq \left\{\,
        \sum_{j=1}^m \lambda_j a_j 
        \mathop{\,\bigg\vert\,}
        \lambda_j\geq 0,\ \sum_{j=1}^m \lambda_j = 1
    \,\right\}.
\]
A classical polytope can be seen as the image of a standard \(m\)-simplex \(\Delta^m\) under the linear map
\({\Delta^m\to\R^n}\) given by the multiplication with a matrix \(A\in\R^{m\times n}\). 

Similarly, a \emph{tropical polytope} \(\tconv(A)\) is the image of a
linear map \(\torus[m]\to\torus\) given by tropical multiplication with a matrix \(A\in\TT^{m\times n}\).
The tropical polytope \(\tconv(A)\) is explicitly given by the set 
\[
    \tconv(a_1,\dots,a_m) \coloneqq \left\{\,
        \bigoplus_{j=1}^m \lambda_j a_j 
        \mathop{\,\bigg\vert\,}
        \lambda_j\in\mathbb{R}
    \,\right\}.
\]
By slight abuse of notation, we denote both the matrix and the set of its columns 
\(\left\{a_1,\dots,a_n\right\}\) by \(A\). 

The identification \(\torus \cong \R^{n-1}\) allows us to compare tropical polytopes with classical polytopes,
but it turns out that a tropical polytope is, in general, not classically convex; see \Cref{fig:Skinny}. This motivates the definition of a \emph{polytrope}, which is a tropical polytope that is also a classical polytope \cite{JK:2010}.

Being both tropically and classically convex reveals much about the structure of a polytope. 
For example, for a given polytrope \(P\subseteq\torus\), \citeauthor{DS:2004} \cite{DS:2004} show that there exists a tropical square matrix \(C\in\TT^{n\times n}\) such that \(P\) is given by \[
    P = \Q(C) \coloneqq \left\{\, 
        x\in\torus 
        \mid
        x_i - x_j \geq c_{ij}
    \,\right\}.
\] In fact, it has been shown that above matrix \(C\) satisfies the relation \[
    \tconv(C^*) = \Q(C^*) = \Q(C)
\] where \(C^* = I_n\oplus C \oplus C^{\odot2} \oplus \dots \oplus C^{\odot (n-1)}\) is the \emph{Kleene star} of \(C\) \cite{deLaPuente:2013} and the identity matrix for the tropical matrix product is given by \[
    I_n = \begin{pmatrix}
     0 & -\infty & \dots & -\infty \\
     -\infty & 0 & \ddots & \vdots \\
     \vdots & \ddots & \ddots & -\infty \\
     -\infty & \dots & -\infty & 0
    \end{pmatrix}.
\]

\begin{example}\label{ex:polytrope}
    \begin{figure}
        \centering
        \pgfdeclarelayer{background}
\pgfdeclarelayer{nodelayer}
\pgfdeclarelayer{edgelayer}
\pgfsetlayers{background, edgelayer, nodelayer}

\begin{tikzpicture}[
        line width=1pt,
        roundnode/.style={circle, draw=black, fill=black, thin, inner sep=1pt},
        pseudonode/.style={circle, draw=black, fill=white, thin, inner sep=1pt},
        scale=0.8
    ]
    \begin{pgfonlayer}{background}
        \draw[help lines, color=black, dotted] (-6,-1) grid (1,6);
        \draw[->,ultra thick] (-6,0)--(1,0) node[right]{$x_2 - x_1$};
        \draw[->,ultra thick] (0,-1)--(0,6) node[above]{$x_3 - x_1$};
    \end{pgfonlayer}
    \begin{pgfonlayer}{nodelayer}
        \node[pseudonode] (a1) at (-1,5) {};
        \node[roundnode] (a2) at (-4,5) {};
        \node[pseudonode] (a3) at (-3,0) {};
        \node[roundnode] (b1) at (-4,0) {};
        \node[roundnode] (b2) at (-1,2) {};
    \end{pgfonlayer}
    \begin{pgfonlayer}{edgelayer}
        \draw[black,thick,fill=blue,opacity=0.3] (a1.center) -- (a2.center) -- (b1.center) -- (a3.center) -- (b2.center) -- cycle;

        \draw[orange!50, thick] (-6,3) -- (-3,6);
    \end{pgfonlayer}
    





\end{tikzpicture}
        \caption{The polytrope from \Cref{ex:polytrope}. The three black dots are the tropical vertices while there
        are two additional pseudovertices}
        \label{fig:polytrope}
    \end{figure}
    
    Consider the matrix \[
        C = \begin{pmatrix}
            -1 & 1 & -5\\
            -4 & 0 & -\infty\\
             0 & 3 & -\infty
        \end{pmatrix}
    \ \text{with Kleene star}\ 
        C^* = \begin{pmatrix}
            0  & 1 & -5\\
            -4 & 0 & -9\\
            0  & 3 &  0
        \end{pmatrix}.
    \] The resulting polytrope is given by the inequalities \begin{align*}
        x_2 - x_1 &\leq -1,& x_3 - x_1 &\leq 5,& x_3 - x_2 &\geq 3, \\
        x_2 - x_1 &\geq -4,& x_3 - x_1 &\geq 0,& x_3 - x_2 &\leq 9. \qedhere
    \end{align*} 
\end{example}

While the conversion of a facet description ($h$-description) into a vertex description ($v$-description) 
of a classical polytope is difficult in general (even exponential in some cases \cite{Bremner:1999}),
we can obtain a subset of the classical vertices of a polytrope using the Floyd--Warshall algorithm in time \(\algO(n^3)\).

To calculate all pseudovertices---that is, all vertices of a polytrope as a classical polytope---we use the 
following result to calculate the pseudovertices on all line segments between 
all pairs of tropical vertices \cite{DS:2004}.

\begin{theorem}[{\cite[Prop.\ 3]{DS:2004}}]\label{thm:breakpoints-trop-line}
    The tropical line segment between two points \(x,y\in\torus\) is the union of at most \(n-1\) ordinary line segments.
\end{theorem}
\begin{corollary}
    Given a matrix \(C\in\TT^{n\times n}\) with \(P = \Q(C)\), the classical vertex description of \(P\)
    may be calculated in \(\algO(n^3\log{n})\) time.
    \begin{proof}
        The procedure of calculating the breakpoints of a tropical line segment as in \Cref{thm:breakpoints-trop-line}
        requires sorting the vector of differences \(x-y\), which can be done in \(\algO(n\log{n})\) time.
        To find all pseudovertices of \(P\), we need to calculate the breakpoints of all line segments between
        all pairs of tropical vertices.

        The tropical vertices of \(P\) are obtained by computing the Kleene star \(C^*\) which can be done 
        in \(\algO(n^3)\) by the Floyd--Warshall algorithm.
        Then, we need to compute the breakpoints of \(n^2\) tropical line segments which involves 
        a sorting operation of \(\algO(n\log{n})\). This means the total time required to compute all breakpoints 
        is \(\algO(n^3\log{n})\).
    \end{proof}
\end{corollary}

It turns out that the metric balls for the tropical metric are polytropes. We give a variant of the 
construction of the corresponding matrix \(C\) due to \citeauthor{Joswig:ETC} \cite[Example 6.39]{Joswig:ETC}
as it will prove useful for explicit computations.
\begin{lemma}\label{lem:trop-ball-polytrope}
    Let \(y\in\torus\) and \(r\geq 0\). The tropical ball \(\TropBall{y}{r}\) is a polytrope.

    \begin{proof}
        First, assume that \(y = (0,\dots,0)\). Then \(x\in\torus\) satisfying \[
            \dtr(x,0) = \max_{1\leq j < i \leq {n}}\big|x_i-x_j\big| \leq r
        \] is equivalent to \begin{equation}\label{eq:polytrope-ineq}
            x_i-x_j \geq -r,\quad\text{and}\quad x_j - x_i \geq -r,
        \end{equation} which is exactly the polytrope given by the matrix \[
            C = \begin{pmatrix}
                0 & -r & \dots & -r \\
                -r & 0 & \ddots & \vdots \\
                \vdots & \ddots & \ddots & -r \\
                -r & \dots & -r & 0
            \end{pmatrix}.
        \] 
        Now, we can reduce the case of arbitrary center \(y\in\torus\) to the above case by translation \[
            \TropBall{y}{r} - y = \TropBall{0}{r}.
        \] Inspecting the inequalities in \eqref{eq:polytrope-ineq} after applying this translation 
        gives the matrix \(C\) with entries \[
            c_{ij} = -r + y_i - y_j
        \] representing \(\TropBall{y}{r}\) as a polytrope.
    \end{proof}
\end{lemma}


\section{Defining Tropical Fr\'{e}chet Means}

With the setup and relevant results outlined in the previous section, we now define our main object of interest. A tropical Fr\'{e}chet mean of $P$ is a point $\tfm$ that minimizes the sum of squares of the tropical distances from $\tfm$ to every point in $P$. 

\begin{definition}
For a finite set $P$ of points $p^{(1)}, p^{(2)},\ldots, p^{(m)} \in \torus$,  any point $\tfm \in \torus$
minimizing the objective function 
\begin{equation}\label{eq:frechet}
  c(x) \coloneqq \sum_{\nu=1}^{m}{\dtr^2(\tfm,p^{(\nu)})}
\end{equation}
is a {\em tropical Fr\'echet mean} of $P$. 

\end{definition}

First, we note that any finite set \(P\subset\torus\) indeed admits a tropical Fr\'echet mean.

\begin{lemma}\label{Lem:existence}
For any finite set $P$ of points $p^{(1)}, p^{(2)},\ldots, p^{(m)} \in \torus$, there exists at least one tropical Fr\'echet mean 
$\tfm \in \torus$.

\begin{proof}
    The tropical projective torus equipped with the tropical metric is a proper metric space 
    (i.e., every closed, bounded subspace is compact) \cite{MOP:2014}.
    Thus, the existence of the tropical Fréchet mean is guaranteed \cite{ohta:2012}. 
\end{proof}
\end{lemma}

We now show a convexity result for the tropical distance---namely, that the squared tropical distance 
to a fixed point is a convex function. In other words, the squared tropical distance of a convex combination 
is less than or equal to the convex combination of the squared tropical distances.
    
\begin{lemma}\label{lem:convexity}
    Let $x,y,p \in \torus$. Set $d_{\lambda}=\dtr(\lambda y+(1-\lambda)x, p)$
    for $0\le \lambda \le 1$. Then
    \begin{enumerate}[(i)]
        \item $d_{\lambda}^{2} \leq \lambda \cdot d_{1}^{2} + (1-\lambda) \cdot d_{0}^{2}$; and 
        \item If equality in (i) holds for $0<\lambda<1$, then $d_{0}=d_{1}$.
    \end{enumerate}

    \begin{proof}
        Since the tropical metric is invariant under translation, we may assume $p=\tilde{0}$. 
        Then, by Definition \ref{def:tropicalmetric},
        \begin{equation*}
            d_{\lambda}=\max_{1\le i\le n}(\lambda y_i +(1-\lambda) x_i)-\min_{1\le i\le n}(\lambda y_i +(1-\lambda) x_i).
        \end{equation*}
        For given $\lambda$, there exist indices $i, i'$ such that
        \begin{equation}
        \begin{split}
            d_{\lambda} &=(\lambda y_i+(1-\lambda) x_i)-(\lambda  	  y_{i'} +(1-\lambda) x_{i'}) \\
            &=\lambda(y_i-y_{i'})+(1-\lambda)(x_i-x_{i'}).
        \end{split}
        \end{equation}
     
        Note that $|y_i - y_{i'}|\le d_{1}$ and $|x_i - x_{i'}|\le d_{0}$, so we get the following inequality:
        \begin{equation}
        \begin{split}
            d_{\lambda}^{2} &= (\lambda(y_i-y_{i'})+(1-\lambda)(x_i-x_{i'}))^{2} \\
            & \le (\lambda d_{1}+(1-\lambda)d_{0})^{2} \\      
            & = \lambda^{2} d_{1}^{2} + (1-\lambda)^{2} d_{0}^{2} + \lambda(1-\lambda) \cdot 2d_{1}d_{0} \\
            & \le \lambda^{2} d_{1}^{2} + (1-\lambda)^{2} d_{0}^{2} + \lambda(1-\lambda)(d_{1}^{2}+d_{0}^{2}) \\
            & = \lambda d_{1}^{2} + (1-\lambda) d_{0}^{2},
        \end{split}
        \end{equation}
        which proves (i). For (ii), suppose $0<\lambda<1$ and the equality in (i) holds. 
        Then $\lambda(1-\lambda)\cdot 2d_{1}d_{0}=\lambda(1-\lambda)(d_{1}^{2}+d_{0}^{2})$, thus $(d_{1}-d_{0})^{2}=0$, so (ii) holds.
    \end{proof}
\end{lemma}

It turns out that knowing one Fr\'echet mean for a finite set \(P\) already provides very significant information on all other possible 
Fr\'echet means for \(P\).
This results in the following characterization in terms of the tropical distance.
    
\begin{theorem}\label{thm:frechet}
     For a finite set $P$ of points $p^{(1)}, p^{(2)},\ldots, p^{(m)} \in \torus$, there exist constants 
     $d_{1},d_{2},\ldots,d_{m}\ge 0$ such that for any Fr\'echet mean $\tfm$ of $P$ and $1\le\nu\le m$, 
     we have $$\dtr(\tfm,p^{(\nu)})=d_{\nu}.$$

    \begin{proof}
        If $P$ has a unique Fr\'echet mean, the claim is immediate. Otherwise, suppose $\tfm$ and $\tfm[y]$ 
        are two different Fr\'echet means of $P$. This means
        \[
            C \coloneqq \sum_{\nu=1}^{m}{\dtr^2(\tfm,p^{(\nu)})}=\sum_{\nu=1}^{m}{\dtr^2(\tfm[y],p^{(\nu)})},
        \]
        is minimal among all points in $\R^{n}/\R\one$. Hence, for any $0<\lambda<1$, we have
        \begin{equation*}
            \sum_{\nu=1}^{m}{\dtr^2(\lambda \tfm[y] + (1-\lambda)\tfm,\ p^{(\nu)})}\ge C.
        \end{equation*}
        By \Cref{lem:convexity}.(i),  we get for $1\le\nu\le m$ that
        \begin{equation}\label{eq:individual}
            \dtr^2(\lambda \tfm[y] + (1-\lambda) \tfm,\ p^{(\nu)}) 
            \le \lambda\cdot \dtr^2(\tfm[y],p^{(\nu)}) +(1-\lambda) \cdot \dtr^2(\tfm,p^{(\nu)}),
        \end{equation}
        and summing over $j$ gives
        \begin{equation*}
            \sum_{\nu=1}^{m}\dtr^2(\lambda \tfm[y] + (1-\lambda) \tfm,\ p^{(\nu)})\le \lambda C + (1-\lambda) C = C.
        \end{equation*}
        
        Hence all inequalities in \eqref{eq:individual} are equalities. Using \Cref{lem:convexity}.(ii) we obtain 
        $$\dtr(\tfm,p^{(\nu)})=\dtr(\tfm[y],p^{(\nu)})$$ for $1\le\nu\le m$. This means the constants \(d_\nu\) are given by 
        $d_\nu=\dtr(\tfm,p^{(\nu)})$.
    \end{proof}
\end{theorem}

\subsection{The Fréchet Mean Polytrope}
As a consequence of \Cref{thm:frechet}, the set \(\overline{P}\) of all possible Fr\'echet means lies in the intersection
of tropical balls. 
We can actually show the converse and characterize \(\overline{P}\) as 
this intersection of metric balls. This justifies the following terminology.
\begin{definition}
    For a finite set \(P\subset\torus\) of points, we denote the set of Fr\'echet means of $P$ by $\overline{P}$ 
    and call it the \emph{Fr\'echet mean polytrope} (or \emph{FM polytrope} for short) of \(P\).
\end{definition}

\begin{theorem}\label{cor:polytope}
    For a finite set $P \subset \R^n/\R\one$, the set $\overline{P}$ is a polytrope in $\R^n/\R\one \cong \R^{n-1}$.
    In particular, \(\overline{P}\) is bounded in \(\torus\).

    \begin{proof}
         From \Cref{thm:frechet}, we get constants $d_{1},\ldots,d_{m}$ such that every Fr\'echet mean $\tfm$ of $P$ 
         satisfies $\dtr(\tfm,p^{(\nu)})=d_\nu$ for $1\le\nu\le m$.
         Then $\tfm$ belongs to the closed balls \(\TropBall{p^{(\nu)}}{d_\nu}\) for every \(1\leq\nu\leq m\),
         meaning \(\overline{P} \subseteq \bigcap_{\nu=1}^m \TropBall{p^{(\nu)}}{d_\nu}\).
         
         Conversely, if a point $\tfm[y]$ satisfies the inequalities ${\dtr(\tfm[y],p^{(\nu)})\le d_{j}}$ for \(1\leq j\leq m\),
         then
         \begin{equation}\label{eq:ineq}
            \sum_{j=1}^{m}{\dtr^2(\tfm[y],p^{(\nu)})}\le \sum_{j=1}^{m}{d_{j}^{2}}.
         \end{equation} 
         The sum of squares on the right hand side of \eqref{eq:ineq} is minimal, so $\tfm[y]$ is already a Fr\'echet mean of $p$,
         meaning \({\overline{P} = \bigcap_{\nu=1}^m \TropBall{p^{(\nu)}}{d_\nu}}\).
         
         Finally, every closed ball centered at $p^{(\nu)}$ with radius $d_\nu$ is defined by finitely 
         many linear inequalities of the form \(x_i - x_j \leq c^{(\nu)}_{ij}\) and it is bounded in $\R^{n-1}$.
         Thus, the intersection of finitely many bounded polytropes is still a bounded polytrope, 
         which completes our proof.
    \end{proof}
\end{theorem}

\begin{figure}
\centering
\tikzstyle{red dot}=[fill=red, draw=black, shape=rectangle, mark size=2pt]
\tikzstyle{vertex}=[circle, draw=black!85, fill=black!85, thin, scale=0.5]

\definecolor{maxoid_1}{RGB}{51,34,136}
\definecolor{maxoid_2}{RGB}{17,119,51}
\definecolor{maxoid_3}{RGB}{68,170,153}
\definecolor{maxoid_4}{RGB}{136,204,238}
\definecolor{maxoid_5}{RGB}{221,204,119}
\colorlet{maxoid_6}{orange!65}
\definecolor{maxoid_7}{RGB}{204,102,119}
\definecolor{maxoid_8}{RGB}{170,68,153}
\definecolor{maxoid_9}{RGB}{136,34,85}

\tikzstyle{blue_edge}=[-, draw=maxoid_1, fill=maxoid_1, fill opacity=0.2]
\tikzstyle{orange_edge}=[-, draw=maxoid_6, fill=maxoid_6, fill opacity=0.2]
\tikzstyle{green_edge}=[-, draw=maxoid_7, fill=maxoid_7, fill opacity=0.2]
\tikzstyle{purple_edge}=[-, draw=maxoid_3, fill=maxoid_3, fill opacity=0.2]
\tikzstyle{red_edge}=[ultra thick, draw=orange, line cap=round]

\pgfdeclarelayer{background}
\pgfdeclarelayer{nodelayer}
\pgfdeclarelayer{edgelayer}
\pgfsetlayers{background, edgelayer, nodelayer}
\begin{tikzpicture}[scale=0.3]
    \begin{pgfonlayer}{background}
        \draw[help lines, color=black, dotted] (-9,-7) grid (19,17);
        \draw[->,ultra thick] (-9,0)--(19,0) node[right]{$x_2-x_1$};
        \draw[->,ultra thick] (0,-7)--(0,17) node[above]{$x_3-x_1$};
    \end{pgfonlayer}
	\begin{pgfonlayer}{nodelayer}
    \node[vertex,label={[text=black,above left, align=left]:(0,0,8)}] (0) at (0, 8) {};
    \node[vertex,label={[text=black,above right, align=left]:(0,10,2)}] (2) at (10, 2) {};
		\node[vertex,label={[text=black,above,yshift=0mm, align=left]:(0,2,4)}] (3) at (2, 4) {};
		\node[vertex,label={[text=black,below,yshift=-2mm, align=left]:(0,5,3)}] (4) at (5, 3) {};
		\node (5) at (0, 0) {};
		\node (6) at (8, 8) {};
		\node (7) at (-8, 0) {};
		\node (8) at (-8, 8) {};
		\node (9) at (0, 16) {};
		\node (10) at (8, 16) {};
		\node (13) at (2, 2) {};
		\node (14) at (4, 4) {};
		\node (19) at (0, 2) {};
		\node (20) at (0, 4) {};
		\node (21) at (2, 6) {};
		\node (22) at (4, 6) {};
		\node (23) at (10, 10) {};
		\node (24) at (18, 2) {};
		\node (25) at (18, 10) {};
		\node (26) at (10, -6) {};
		\node (27) at (2, -6) {};
		\node (28) at (3, 3) {};
		\node (29) at (5, 5) {};
		\node (30) at (7, 5) {};
		\node (31) at (7, 3) {};
		\node (32) at (5, 1) {};
		\node (33) at (3, 1) {};
	\end{pgfonlayer}
	\begin{pgfonlayer}{edgelayer}
		\draw [style={green_edge}] (10.center)
			 to [in=90, out=-90] (6.center)
			 to (5.center)
			 to (7.center)
			 to (8.center)
			 to (9.center)
			 to cycle;
        \draw [style={blue_edge}] (19.center)
			 to (20.center)
			 to (21.center)
			 to (22.center)
			 to (14.center)
			 to (13.center)
			 to cycle;
		\draw [style={orange_edge}] (27.center)
			 to (26.center)
			 to (24.center)
			 to (25.center)
			 to (23.center)
			 to (13.center)
			 to cycle;
		\draw [style={purple_edge}] (29.center)
			 to (28.center)
			 to (33.center)
			 to (32.center)
			 to (31.center)
			 to (30.center)
			 to cycle;
		\draw [style={red_edge}] (14.center) to (28.center);
	\end{pgfonlayer}
\end{tikzpicture}    
\caption{%
This figure shows the intersection of balls centered at ${P = \{(0,0,8), (0,2,4), (0,5,3), (0,10,2)\}}$. 
The associated Fréchet mean polytrope \(\overline{P}\) is the line segment from $(0,3,3)$ and $(0,4,4)$ 
and the minimal sum of squares is 34.
}
\label{fig:intersect2}
\end{figure}

\subsection{Case Study: Two Points}

We first investigate the case of computing the tropical Fr\'echet mean for two points in $\R^n/\R\one$, thus the sample size is $m = 2$. 
It turns out these Fréchet means are closely related to tropical geodesics.
\begin{definition}
For points $x,y \in \R^n/\R\one$, the {\em tropical line segment} between them is
$$
\Gamma^{\mathrm{tr}}_{x, y} 
\coloneqq \{\,
    a \odot x \oplus b \odot y \mid a, b \in \R 
\,\}\subset\torus.
$$

We say that \(m\in\Gamma^{\mathrm{tr}}_{x, y}\) is the \emph{midpoint} of \(x\) and \(y\) if it is equidistant to \(x\) and \(y\).
\end{definition}
Tropical line segments are geodesics, but these are not the only geodesics; 
in fact, there are infinitely many geodesics between any two given points \cite{MLYK:2022}.

\begin{proposition}
The midpoint of a tropical geodesic $\Gamma_{p_1, p_2}$ between two points $p_1,p_2 \in \R^n/\R\one$ 
is a tropical Fr\'echet mean. 
Conversely, every tropical Fr\'echet mean of $P= \{ p_1, p_2\}$ is the midpoint of some geodesic $\Gamma_{p_1,p_2}$.
\end{proposition}
\begin{proof}
Let $\ell = \dtr(p_1,p_2)$. For any point $x \in \R^n/\R\one$, we have
\begin{equation}\label{eq:cauchy}
\begin{split}
    \dtr(x, p_1)^{2} + \dtr(x, p_2)^{2} 
    &\geq \frac{\left(\dtr(x, p_1) + \dtr(x, p_2)\right)^{2}}{2} \\
    &\ge \frac{\dtr(p_1,p_2)^{2}}{2} = \frac{\ell^{2}}{2}.
\end{split}
\end{equation}

If $m$ is the midpoint of $\Gamma_{p_1, p_2}$, then 
\[\dtr(m, p_1) = \dtr(m, p_2) = \frac{\ell}{2}. \]
Therefore \eqref{eq:cauchy} is an equality and $m$ is a Fr\'echet mean of $P$.

Conversely, if $\tfm$ is a Fr\'echet mean of $P$, then all equalities in \eqref{eq:cauchy} must hold. 
Hence $\dtr(\tfm, p_1) + \dtr(\tfm, p_2) = \ell$ and $\dtr(\tfm, p_1) = \dtr(\tfm, p_2)$, 
which means $\dtr(\tfm, p_1) = \dtr(\tfm, p_2) = \frac{\ell}{2}$.
By definition of $\dtr$, there exists a tropical geodesic from $p_{i}$ to $\tfm$ with length $\frac{\ell}{2}$ for $i=1,2$, 
so their concatenation is a tropical geodesic from $p_{1}$ to $p_{2}$ and $\tfm$ is its midpoint.
\end{proof}

\section{Computing Tropical Fr\'{e}chet Means}
We continue by discussing how to compute tropical Fr\'echet means by solving the optimization problem \eqref{eq:frechet}.
As a consequence of \Cref{cor:polytope}, we may compute the FM polytrope for a finite set $P \subset \R^{n}/\R\one$ in the following manner:
\begin{enumerate}
    \item Compute one Fr\'echet mean \(\tfm\);
    \item Compute the constants $d_\nu$ from Theorem \ref{thm:frechet} from \(\tfm\);
    \item Compute the intersection of the tropical balls \(\TropBall{p^{(\nu)}}{d_\nu}\) centered at each $p^{(\nu)} \in P$ 
        with radius $d_\nu$.
\end{enumerate}

From \Cref{lem:trop-ball-polytrope}, we know that each \(\TropBall{p^{(\nu)}}{d_\nu}\) is a polytrope, which 
simplifies the calculation of the intersection in Step (3). In that case, the FM polytrope \(\overline{P}\) itself is 
a polytrope represented by the square matrix \(\overline{C}\in\R^{n\times n}\) with entries \begin{equation}\label{eq:fm-polytrope-h-desc}
    \overline{c}_{ij} = \max_{1\leq\nu\leq m} -d_\nu + p_i^{(\nu)} - p_j^{(\nu)}.
\end{equation}

\begin{proposition}\label{prop:fm-polytrope-compute}
    Given a tropical Fréchet mean \(\tfm\) for a finite set \(P = \{\, p^{(1)}, \dots, p^{(m)}\,\}\subset\torus\),
    a facet description of the associated Fréchet mean polytrope \(\overline{P}\) can
    be computed in \(\algO(mn^2)\) time. A 
    tropical vertex description of \(\overline{P}\) can be computed in 
    \(\algO(mn^2 + n^3)\) time and the classical vertex description may be computed in
    \(\algO(mn^2 + n^3\log{n})\) time.

    \begin{proof}
        Computing the facet description of \(\overline{P}\) using \eqref{eq:fm-polytrope-h-desc} means we need to
        find the minimum of \(m\) values, which can be done in linear time. Finding the facet description 
        in terms of the matrix \(\overline{C}\) requires \(\algO(n^2)\)-many operations, resulting in a 
        \(\algO(mn^2)\) runtime.

        To obtain a tropical vertex description from \(\overline{C}\), we need to compute the Kleene star \(\overline{C}^*\),
        for which we can use the Floyd--Warshall algorithm in time \(\algO(n^3)\), 
        resulting in a total runtime of \(\algO(mn^2 + n^3)\) when starting from a tropical Fréchet mean.
        
        Finally, computing the classical vertex description using \Cref{thm:breakpoints-trop-line} requires 
        \(\algO(n^3\log{n})\) time, totaling in \(\algO(mn^2 + n^3\log{n})\) time.
    \end{proof}
\end{proposition}

A clear advantage to this approach is that the computation of the facet description is \emph{embarrassingly parallelizable}.
Thus, it remains to compute an initial Fréchet mean for Step (1). This itself is an intricate matter given its nature
as a quadratic optimization problem, which we now discuss.

\subsection{Exact Quadratic Optimization}\label{sec:exact-quadratic}
In order to compute an initial Fr\'echet mean directly, we need the exact solution to a piecewise quadratic 
minimization problem. 
We may break down this problem since the objective function is always the maximum of finitely many quadratic functions 
over $\mathbb{R}$. Thus, we can instead solve the following quadratic program to obtain a tropical Fr\'echet mean:
\begin{optimproblem}
  \objectivefunction{$\sum_{\nu=1}^m d_\nu^2$}
  \variables{$x_i,d_\nu\in\R, i\in[n],\nu\in[m]$}
  \constraints{$d_\nu^2 \geq (x_i - x_j - p_{ij}^{(\nu)})^2$ for \(i\neq j\in[n]\) and \(\nu\in[m]\)}
\end{optimproblem}\noindent
Here, we include the constants \(d_\nu\) from \Cref{thm:frechet} in the quadratic program to solve this optimization globally.
We illustrate this idea with the following example.

\begin{example}\label{ex:quadratic-optimisation}
\begin{figure}
\centering
\tikzstyle{vertex}=[circle, draw=black!85, fill=black!85, thin, scale=0.5]

\definecolor{maxoid_1}{RGB}{51,34,136}
\definecolor{maxoid_2}{RGB}{17,119,51}
\definecolor{maxoid_3}{RGB}{68,170,153}
\definecolor{maxoid_4}{RGB}{136,204,238}
\definecolor{maxoid_5}{RGB}{221,204,119}
\colorlet{maxoid_6}{orange!65}
\definecolor{maxoid_7}{RGB}{204,102,119}
\definecolor{maxoid_8}{RGB}{170,68,153}
\definecolor{maxoid_9}{RGB}{136,34,85}

\tikzstyle{red dot}=[fill=red, draw=black, shape=rectangle, mark size=2pt]
\tikzstyle{blue_edge}=[-, draw=maxoid_1, fill=maxoid_1, fill opacity=0.2]
\tikzstyle{orange_edge}=[-, draw=maxoid_3, fill=maxoid_3, fill opacity=0.2]
\tikzstyle{green_edge}=[-, draw=maxoid_5, fill=maxoid_5, fill opacity=0.2]
\tikzstyle{purple_edge}=[-, draw=maxoid_6, fill=maxoid_6, fill opacity=0.2]
\tikzstyle{red_edge}=[ultra thick, draw=orange, line cap=round]

\pgfdeclarelayer{background}
\pgfdeclarelayer{nodelayer}
\pgfdeclarelayer{edgelayer}
\pgfsetlayers{background, edgelayer, nodelayer}

\begin{tikzpicture}[scale=0.3]
    \begin{pgfonlayer}{background}
        \draw[help lines, color=black, dotted] (-14,-24) grid (12,8);
        \draw[->,ultra thick] (-14,0)--(12,0) node[right]{$x_2 - x_1$};
        \draw[->,ultra thick] (0,-24)--(0,8) node[above]{$x_3 - x_1$};
    \end{pgfonlayer}
	\begin{pgfonlayer}{nodelayer}
    \node[vertex,label={[text=black,above right, align=left]:(-3,0,0)}] (0) at (3,3) {};
    \node[vertex,label={[text=black,above right, align=left]:(0,-6,0)}] (1) at (-6,0) {};
		\node[vertex,label={[text=black,above,yshift=0mm, align=left]:(0,0,-12)}] (2) at (0,-12) {};
		\node (3) at (3, -1) {};
		\node (4) at (-1, 3) {};
		\node (5) at (-1, -1) {};
		\node (6) at (3, 7) {};
		\node (7) at (7, 7) {};
		\node (8) at (7, 3) {};
		\node (9) at (1, 0) {};
		\node (10) at (-6, 7) {};
		\node (11) at (1, 7) {};
		\node (12) at (-13, 0) {};
		\node (13) at (-6, -7) {};
		\node (14) at (-13, -7) {};
		\node (15) at (0, -1) {};
		\node (16) at (-11, -12) {};
		\node (17) at (11, -1) {};
		\node (18) at (11, -12) {};
		\node (19) at (0, -23) {};
		\node (20) at (-11, -23) {};
		\node [style=red dot] (21) at (0, -1) {};
		\node (22) at (3, -2) {\textcolor{red}{(0,0,-1)}};
	\end{pgfonlayer}
	\begin{pgfonlayer}{edgelayer}
		\draw [style={blue_edge}] (10.center)
			 to (11.center)
			 to (9.center)
			 to (13.center)
			 to (14.center)
			 to (12.center)
			 to cycle;
		\draw [style={orange_edge}] (4.center)
			 to (6.center)
			 to (7.center)
			 to (8.center)
			 to (3.center)
			 to (5.center)
			 to cycle;
		\draw [style={green_edge}] (17.center)
			 to (18.center)
			 to (19.center)
			 to [in=360, out=180] (20.center)
			 to (16.center)
			 to (15.center)
			 to cycle;
	\end{pgfonlayer}
\end{tikzpicture}
\caption{%
The tropical Fréchet mean is an intersection of tropical spheres. 
We visualise points in $\mathbb{R}^3/\mathbb{R}\mathbf{1}$ by taking the representative with zero 
as its first coordinate. 
The figure shows the unique Fréchet mean $(0,0,-1)$ of the points from Example \ref{ex:quadratic-optimisation}. 
}
\label{fig:intersect}
\end{figure}

	Let $n=m=3$ and \[P=\{p^{(1)},p^{(2)},p^{(3)}\}=\{(-3,0,0),(0,-6,0),(0,0,-12)\}.\] We claim that 
    $(0,0,-1)\in \overline{P}$ and the corresponding minimal sum of squares is $186$:
	\[
	\begin{split}
		\dtr^{2}((0,0,-1),p^{(1)}) &= \left[(-3-0)-(0-(-1))\right]^{2}=4^{2}=16, \\
		\dtr^{2}((0,0,-1),p^{(2)}) &= \left[(-6-0)-(0-(-1))\right]^{2}=7^{2}=49, \\
		\dtr^{2}((0,0,-1),p^{(3)}) &= \left[(-12-(-1))-(0-0)\right]^{2}=11^{2}=121, \\
		\quad \implies 16+49+121&=186.
	\end{split}
	 \]
	 
	For any $x \in \R^3/\R\one$, the squared tropical distances between \(x\)
    and each point \(p_i\in P\) admit the following expressions.
	\[
	\begin{split}
		\dtr^{2}(x,p^{(1)}) &= \max\left((x_2-x_1-3)^2, (x_3-x_1-3)^2, (x_3-x_2)^2\right), \\
		\dtr^{2}(x,p^{(2)}) &= \max\left((x_2-x_1+6)^2, (x_3-x_1)^2, (x_3-x_2-6)^2\right), \\
		\dtr^{2}(x,p^{(3)}) &= \max\left((x_2-x_1)^2, (x_3-x_1+12)^2, (x_3-x_2+12)^2\right). \\
	\end{split}
	\]
	The objective function for the tropical Fréchet mean is the maximum of $27$ quadratic functions 
    in real variables $x_{1},x_{2},x_{3}$ for each possible combination of the components of each 
    \(\dtr^{2}(\tfm,p^{(\nu)})\). Note that each $x_i^2$ has a nonnegative coefficient. 
    
    First, we show that the maximum of these $27$ functions is always at least $186$. 
    To see this, we use the following lower bounds on each \(\dtr^{2}(\tfm,p^{(\nu)})\).
	\begin{equation}\label{eq:lower-bounds}
	\begin{split}
		\dtr^{2}(x,p^{(1)})&\ge (x_3-x_1+3)^2, \\
		\dtr^{2}(x,p^{(2)})&\ge (x_3-x_2+6)^2, \\
		\dtr^{2}(x,p^{(3)})&\ge \frac{4}{11}(x_3-x_1-12)^2+\frac{7}{11}(x_3-x_2-12)^2.
	\end{split}
	\end{equation}
	These lower bounds allow us to estimate the sum of squares as claimed.
	\[
	\begin{split}
		\sum_{\nu=1}^{3}{\dtr^{2}(x,p^{(\nu)})}
            &\ge (x_3-x_1-3)^2+(x_3-x_2-6)^2 +\frac{4}{11}(x_3-x_1+12)^2 +\frac{7}{11}(x_3-x_2+12)^2 \\
    	&= \frac{3}{11}\left[5x_1^2 - 10x_1x_3 +6x_2^2 - 12x_2x_3
                + 11x_3^2 - 10x_1 - 12x_2 + 22x_3 + 693 \right] \\
            &= \frac{3}{11}\left[5(x_3-x_1+1)^{2}+6(x_3-x_2+1)^{2}+682\right] \\
            &\ge \frac{3}{11}\cdot 682 = 186. \qedhere
	\end{split}
	\] 
\end{example}

Surprisingly, it is always possible to find such a certificate. Note that all quadratic functions that appeared in the computation of the tropical Fréchet mean in 
\Cref{ex:quadratic-optimisation} are sums of squares of linear functions, so their homogeneous degree 2 components are also sums of squares.
We show that it suffices to check all lower bounds that may be expressed as in \eqref{eq:lower-bounds}.
Indeed, this is implied by our following general result.

\begin{theorem}\label{thm:conv}
    Let $f_1, \dots, f_\ell \in \mathbb{R}[x_1, \dots, x_n]$ be quadratic polynomials with 
    nonnegative degree 2 homogeneous components $f_{j,2}$. Define a function 
    $\bar{f} \colon \mathbb{R}^n \rightarrow \mathbb{R}$ by \[
        \bar{f}(x) \coloneqq \max_{j \in [\ell]} f_j(x)
    \] and let \(c^* := \min_{x \in \mathbb{R}^n} \bar{f}(x)\).
    Then there exist nonnegative weights $w_1, \dots, w_\ell\in\R$ such that
    \[
        \sum_{j=1}^\ell w_jf_j(x) \geq c^*, \quad \forall x \in \mathbb{R}^n 
    \] and $\sum_{j=1}^\ell w_j = 1$.
    
    \begin{proof}
        Since $f_{j,2}(x) \ge 0 $ for all $x\in\R^n$, 
        we also have ${f_{j,2}(x-a) \geq 0}$ for any $a \in \R^n$. Thus we may assume 
        that $\bar{f}$ is minimal at $0$, meaning $\bar{f}(0)=c^*$. 
    	
    	Now for $1\le j \le \ell$,  write	
    	\begin{equation}
    		\begin{split}
    			f_j(x) &= f_{j,2}(x) + f_{j,1}(x) + f_{j,0}(x) \\
    				   &= f_{j,2}(x) + \sum_{i=1}^{n}{a_{ji}x_{i}} + b_{j},
    		\end{split}
    	\end{equation}	
    	where $a_{ji}, b_{j}\in \R$. Since $f_j(0)\leq \bar{f}(0) = c^*$ and by assumption $f_{j,2}(x)\geq c^*$, we have $b_{j} \le c^*$, 
        and at least one of the $f_j(0)$ assumes the value $c^*$. 
        After reordering, there is an index $j^* \in \{ 1, \dots, \ell\}$ such that 
        \[
            b_{1}=\cdots=b_{j^*} = c^*\quad\text{and}\quad b_{j^*+1},\cdots,b_{\ell}<c^*. 
        \]
        
    	Since the $f_j$ are continuous functions, small perturbations \(x\in\R^n\)
        of \(0\) still satisfy 
    	\begin{equation}\label{eq:perturbation}
    		\bar{f}(x) 
            = \max_{1\le j\le \ell}{f_j(x)} 
            = \max_{1\le j \le j^*}{f_j(x)}.
    	\end{equation}

        \begin{claim*}
            For any vector $v\in \R^{n}$, the homogeneous linear parts of each \(f_j\) satisfy
        	\begin{equation}\label{eq:halfspaces}
        		\max_{1\le j\le j^*} f_{j,1}(v)
        		= \max_{1\le j\le j^*}\sum_{i=1}^{n}a_{ji}v_{i}
                \ge 0. 
        	\end{equation}
        \end{claim*}
	    To see this, suppose there exists $v\in \R^{n}$ such that for every $1\le j\le j^*$, 
        we have \[\sum_{i=1}^{n}{a_{ji}v_{i}} < c^*. \]
	
    	For $1\le i\le n$, we let $x'_{i}=v_{i}\cdot \varepsilon$ for some $\varepsilon>0$.
        Then for $1\le j\le j^*$, we have 
    	\begin{equation*}
    		\begin{split}
    			f_j(x'_{1},\ldots,x'_{n}) 
                &= f_{j,2}(x'_{1},\ldots,x'_{n}) 
                    + \left(\sum_{j=1}^{n}{a_{ji}v_{i}}\right)\cdot \varepsilon
                    + c^*\\
    			&= C_{i,v}\varepsilon^{2} 
                    + \left(\sum_{j=1}^{n}{a_{ji}v_{i}}\right)\cdot \varepsilon
                    + c^*,
    		\end{split}
    	\end{equation*}
    	where $C_{i,v}$ is a constant that only depends on $f_{j,2}$ and $v$. 
        When $\varepsilon>0$ and $|\varepsilon|$ is small enough, \eqref{eq:perturbation} implies  $f_j(x'_{1},\ldots,x'_{n}) < c^*$, 
        hence $\bar{f}(x'_{1},\ldots,x'_{n})< c^*$. 
        This is a contradiction to $0$ being the global minimum of $\bar{f}$
        which proves \eqref{eq:halfspaces}.  
	
    	Now, let $A=(a_{ji}) \in \R^{j^*\times n}$ and $\one=(1,\ldots,1)^{\TT}\in \R^{j^*}$. By \eqref{eq:halfspaces} the matrix inequality
    	\begin{equation}\label{eq:primal}
    		Av\le -c^*\one
    	\end{equation} has no solution $v\in \R^{n}$. 
        Farkas' Lemma \cite[Proposition 1.7]{Ziegler.Polytopes} implies that there
        exists a vector $y\in \R^{j^*}$ such that $y^{\TT}\ge 0$, 
        $y^{\TT}\cdot A=0$, and $y^{\TT}\cdot c^*\one > 0$. 
        Hence $y_{j}\ge 0$, $\sum_{j=1}^{j^*}c^*{y_{j}}>0$ and for $1\le i\le n$, we have
    	\begin{equation}\label{eq:linear}
    		\sum_{j=1}^{j^*}{a_{ji}y_{j}} = 0.
    	\end{equation}
				
    	Finally, we have
    	\begin{equation*}
    		\begin{split}
    			\sum_{j=1}^{j^*}{y_{j}f_j(x)} &= \sum_{j=1}^{j^*}{y_{j}f_{j,2}(x)} + \sum_{j=1}^{j^*}{y_{j}f_{j,1}(x)} + \sum_{j=1}^{j^*}{y_jc^*} \\
    			&= \sum_{j=1}^{j^*}{y_{j}f_{j,2}(x)} + \sum_{j=1}^{j^*}{y_{j}\left(\sum_{i=1}^{n}{a_{ji}x_{i}}\right)} + \sum_{j=1}^{j^*}{y_jc^*} \\
    			&= \sum_{j=1}^{j^*}{y_{j}f_{j,2}(x)} + \sum_{i=1}^{n}{\left(\sum_{j=1}^{j^*}{y_{j}{a_{ji}}} \right) x_{i}} + \sum_{j=1}^{j^*}{y_jc^*} \\
    			&\overset{\eqref{eq:linear}}{=} \sum_{j=1}^{j^*}{y_{j}f_{j,2}(x)} + \sum_{j=1}^{j^*}{y_jc^*}.
    		\end{split}
    	\end{equation*}
		
    	Hence $\sum_{j=1}^{j^*}{y_{j}f_j(x)} \ge \sum_{j=1}^{j^*}{y_jc^*}$ for all $x\in \R^{n}$. 
        So we can normalize the $y_j$ and take 
    	\[w_{j} = \begin{cases}
    		\frac{y_{j}}{\sum_{j=1}^{j^*}y_j} &\text{if }1\le j\le j^*,\text{ and} \\
    		0, &\text{otherwise},
    	\end{cases} \]
    	which concludes the proof.
\end{proof}
\end{theorem}

We remark that a widely-used method for calculating Fréchet means in the classical setting is \emph{Sturm's algorithm} \cite{Sturm:2003}.
However, a key assumption required to implement this algorithm requires the space to be 
nonpositively curved, which is an assumption that fails in the tropical projective torus.

In fact, the curvature behavior of the tropical projective torus is known to be complicated \cite{Amendola.Monod:2021}
since regions of all positive, negative, and undefined curvature exist throughout the tropical projective torus.
This affects the output of Sturm's algorithm even in negatively curved settings which we illustrate 
with the following example, taken from \citeauthor{Matteo:2021}~\cite{Matteo:2021}.
\begin{example}
    Consider the set \[
        P_1:= \{ (0, 0, 0) ,\, (0, 2, 4),\, (0, 5, 1) \}\subset\mathbb{R}^3/\mathbb{R}\mathbf{1}.
    \] \Cref{fig:Skinny} shows the tropical geodesic triangle, where the edges between the points in $P_1$ 
    are tropical line segments; this is a triangle with negative curvature in the sense of 
    Alexandrov \cite{Amendola.Monod:2021}. The Fréchet mean $(0, 2, 1)$ is marked in green; 
    Sturm's algorithm converges after 46274 iterations to an incorrect minimum, marked in red.
    \end{example}

      \begin{figure}
        \centering
        \pgfdeclarelayer{background}
\pgfdeclarelayer{nodelayer}
\pgfdeclarelayer{edgelayer}
\pgfsetlayers{background, edgelayer, nodelayer}

\begin{tikzpicture}[
        line width=1pt,
        roundnode/.style={circle, draw=black, fill=black, thin, inner sep=1.5pt},
        scale=0.7
    ]
    \begin{pgfonlayer}{background}
      \draw[help lines, color=black, dotted] (-1,-1) grid (6,5);
      \draw[->,ultra thick] (-1,0)--(6,0) node[right]{$x_2 - x_1$};
      \draw[->,ultra thick] (0,-1)--(0,5) node[above]{$x_3 - x_1$};
    \end{pgfonlayer}
    \begin{pgfonlayer}{nodelayer}
      \node[roundnode] (v1) at (0,0) {};
      \node[roundnode] (v2) at (2,4) {};
      \node[roundnode] (v3) at (5,1) {};

      \node[roundnode,fill=green] (true) at (2,1) {};
      \node[roundnode,fill=red] (false) at (0.8,0.8) {};
    \end{pgfonlayer}
    \begin{pgfonlayer}{edgelayer}
      \draw (v1) -- (2,2);
      \draw (v2) -- (2,1);
      \draw (v3) -- (1,1);
    \end{pgfonlayer}
\end{tikzpicture}
        \caption{Tropical skinny triangle in the sense of Alexandrov. The correct Fréchet mean is shown in green, whereas Sturm's algorithm converges to a wrong point (in red).}
        \label{fig:Skinny}
    \end{figure}
  
The failure of Sturm's algorithm motivates the need to search for an alternative methods that can compute tropical Fréchet means.

\subsection{Reduced Gradient Algorithm}

The problem of computing an initial tropical Fréchet mean has the following properties that
make it amenable to a \emph{reduced gradient algorithm} \cite{FW:1956}. 
First, the objective function is piecewise quadratic and convex, since
\begin{itemize}
    \item any function \((x_i - x_j + c)^2\) is convex as a composition of convex functions;
    \item \(\dtr(x,p)^2\) for fixed \(p\in\torus\) is convex because it is a maximum over functions
        of the form above; and
    \item sums of convex functions are also convex.
\end{itemize}
Any local minimum of a convex function is then a global minimum \cite[Section 4.2.2]{BV:2004}.
This also guarantees that our proposed greedy \Cref{alg:greedy-frechet} below to compute tropical Fréchet means indeed finds a global minimum. 

\begin{algorithm}
    \caption{A greedy algorithm to compute a tropical Fréchet mean.}
    \label{alg:greedy-frechet}
    
    \begin{algorithmic}[1]
    \Procedure{GreedyFr\'echetMean}{}
    \State \textbf{Input:} data $P = \{p^{(1)}, \dots, p^{(m)}\} \subset \torus$ 
    \State \textbf{Output:} Fréchet mean $\overline{P}$ and minimum sum of squares $S$
    \State Initialize $v \gets \frac{1}{m}(p^{(1)} + \dots + p^{(m)})$ and $S \gets c(v)$
    \State Let $k \gets 1$
    \While {$S$ decreases along some direction \(e_i - e_j\)}
    \State Choose a direction \(e_i - e_j\) along which \(S\) decreases
    \State $v \gets v + \frac{2}{k+2} (e_i - e_j)$
    \State Update $S$ and increment $k$
    \EndWhile
    \State \textbf{return} $\left(v, S\right)$
    \EndProcedure
    \end{algorithmic}
\end{algorithm}

\begin{remark}
    In principle, \Cref{alg:greedy-frechet} is a numerical procedure. But in the case where the input data
    only has entries in \(\mathbb{Q}\), the step size in line 8 has been chosen in such a way that \(v\) retains
    rational coordinates in each step. This approach is similar to the implementation 
    of exact rational arithmetic for Frank-Wolfe type algorithms by \citeauthor{BCP:2021}~\cite{BCP:2021}.
\end{remark}

Using \Cref{thm:conv}, we obtain at least a set containing the FM polytrope \(\overline{P}\), which is still is a compact convex set.

Also, at any given point, the objective function \(c\) is given by a sum of squares of the form 
\((x_i - x_j - c)^2\). This means that the gradient of \(c\) will be a linear combination of vectors
\(e_i - e_j\). For this reason, it suffices to minimize \(c\) along the \((e_i - e_j)\)-directions.  Using these observations, we may employ a modified reduced gradient algorithm
as in \Cref{alg:greedy-frechet} to calculate an initial tropical Fréchet mean. The convergence rate of reduced gradient algorithms is known to be sublinear \cite{DUNN1978432}.

\section{A Type Decomposition for the Objective Function}\label{sec:covdec}
While it is possible to formulate a quadratic program for the tropical Fr\'{e}chet mean, quadratic solvers usually do not support the computation
of exact solutions while simultaneously allowing quadratic constraints.   Currently available software is adequate for a numerical solution, but a different approach is required to obtain a truly exact solution. Thus, we need to break down the problem further, which we study in this section.

Note that the objective function \(c\) a tropical Fr\'echet mean minimizes can be expressed as \[
  c(x) 
  = \sum_{\nu=1}^m \dtr^2\left(x, p^{(\nu)}\right) 
  = \sum_{\nu=1}^m \max_{1\leq i,j\leq d} \left\{\, x_i - x_j - p_{ij}^{(\nu)} \,\right\}^2
  = \max_{i,j\in[n]^m}\sum_{\nu=1}^m \left(x_{i^{(\nu)}} - x_{j^{(\nu)}} - p_{ij}^{(\nu)}\right)^2.
\] This is a piecewise sum-of-squares which consists of at most \(\binom{n}{2}^m\) pieces defined over at most \((2^n\cdot n)^m\) regions.
This estimate is obtained from counting every possible expression of $c$ as a sum of $m$ squares. For a specific choice of sum of squares, we obtain several systems of linear inequalities by making a choice of sign for the absolute value
\(\lvert x_i - x_j\rvert\) for each \(1\leq i\neq j\leq n\).
In reality, the number of regions is far lower, which is a consequence of the sign data. We demonstrate this fact with the
following running example.

\begin{example}
    For \(x\in\torus[2]\), the squared distance \(\dtr^2(x,0)\) to the origin is of the form \[
        \dtr^2(x,0) = \max\left\{\,
            (x_2 - x_1)^2,\, (x_3 - x_1)^2,\, (x_3 - x_2)^2
        \,\right\}.
    \] For a specific square \((x_i - x_j)^2\) to realize the value of \(\dtr^2(x,0)\), we require \[
        (x_i - x_j)^2 > (x_k - x_\ell)^2
    \] for \(1\leq i,j,k,\ell\leq n\) where \(\{\, i,j\,\} \neq \{\, k,\ell\,\}\). This is the case when 
    the absolute values satisfy \[
        \lvert x_i - x_j\rvert > \lvert x_k - x_\ell\rvert.
    \] Since the absolute value is a piecewise linear function, we obtain \(8\) different sets of linear inequalities
    corresponding to each possible sign vector \((s_{21},s_{31},s_{32})\in\{\,\pm1\,\}^3\). Together with the choice
    of square $(x_i - x_j)^2$ realizing $\dtr^2(x,0)$, we have $3\cdot 8 = 24$ systems of linear inequalities to check.
    Yet \Cref{fig:braid-arrangement} shows that there are only 6 distinct regions corresponding to sign vectors.
\end{example}
Proceeding as in the above example, however, is an inefficient way to obtain all pieces of \(c\). In fact,
not all sign vectors result in non-empty regions of \(\torus\). The reason behind this is the 
\emph{braid arrangement}, which plays a role in the characterization of polytropes and also appeared in
the bi-tropical covector decompositions for tropical Fermat--Weber points \cite{SBYM:2024}.

\begin{definition}
  The \emph{braid arrangement} \(\braidarr\) in \(\torus\) is the classical hyperplane arrangement given by hyperplanes of the form \(
    x_i - x_j
  \) for \(1\leq j < i \leq n\). 
    The \emph{braid arrangement \(\braidarr(p)\) at \(p\in\torus\)} is the braid arrangement where the 
  origin $0\in\torus$ has been translated to \(p\).
  The braid arrangement \(\braidarr(S)\) of a finite set $S$ is the union of the braid arrangements at $p\in S$.
\end{definition}

\begin{figure}
  {\small\begin{tikzpicture}[
        line width=1pt,
        roundnode/.style={circle, draw=black, fill=black, thin, inner sep=1pt},
        pseudonode/.style={circle, draw=black, fill=white, thin, inner sep=1pt},
        scale=1.5
    ]

\node[anchor=south] at (0,2) {$x_2-x_1$};
\node[anchor=south west] at (2,2) {$x_3-x_2$};
\node[anchor=west] at (2,0) {$x_3-x_1$};

\node[anchor=north] at (0,-2) {\phantom{$x_2-x_1$}};
\node[anchor=east] at (-2,0) {\phantom{$x_3-x_1$}};

\draw[black] (-2,0) -- (2,0);
\draw[black] (0,-2) -- (0,2);
\draw[black] (-2,-2) -- (2,2);

\node[blue] (ppp) at (.5,1.2) {$+++$};
\node[orange,anchor=south,xshift=2pt,yshift=5pt] at (ppp) {$(x_3 - x_1)^2$};

\node[blue] (ppm) at (1.5,.3) {$++-$};
\node[orange,anchor=south,xshift=2pt,yshift=5pt] at (ppm) {$(x_2 - x_1)^2$};

\node[blue] (pmm) at (1,-.3) {$+--$};
\node[orange,anchor=north,xshift=2pt,yshift=0pt] at (pmm) {$(x_3 - x_2)^2$};

\node[blue] (mmm) at (-.6,-1.2) {$---$};
\node[orange,anchor=north,xshift=2pt,yshift=0pt] at (mmm) {$(x_3 - x_1)^2$};

\node[blue] (mmp) at (-1.5,-.3) {$--+$};
\node[orange,anchor=north,xshift=2pt,yshift=0pt] at (mmp) {$(x_2 - x_1)^2$};

\node[blue] (mpp) at (-1,.3) {$-++$};
\node[orange,anchor=south,xshift=2pt,yshift=5pt] at (mpp) {$(x_3 - x_2)^2$};
\end{tikzpicture}}
  \caption{Braid arrangement in the tropical affine plane \(\torus[2]\). The chambers are marked with their covectors
           and types.}
  \label{fig:braid-arrangement}
\end{figure}

We are interested in the chambers of \(\braidarr(S)\), that is, the regions in the complement
of \(\braidarr(S)\).
Each hyperplane $x_i - x_j$ partitions the space into half-spaces where \[
  x_i - x_j > 0 \quad\text{and}\quad x_i - x_j < 0.
\] Thus, each chamber of the braid arrangement corresponds to a choice of sign of \(\lvert x_i - x_j\rvert\)
for each \(1\leq j<i\leq n\). Such a choice of signs \((s_{ij})\in\{\pm\}^N\) is referred to as a \emph{covector}.
\begin{example}
    The braid arrangement in the tropical affine plane \(\torus[2]\) is shown in \Cref{fig:braid-arrangement}.
    It consists of the three classical hyperplanes \[
      x_2 - x_1,\ x_3 - x_1,\text{ and } x_3 - x_2.
    \] This hyperplane arrangement subdivides \(\torus[2]\) into six chambers, each of which corresponds to
    a choice of linear ordering on \(\{1,2,3\}\). We can represent such an ordering by a permutation \(\sigma\in\Sym_3\),
    in which case such an ordering gives rise to a covector \((s_{21},s_{31},s_{32})\). 
    Each chamber in \Cref{fig:braid-arrangement} is marked with the corresponding covector.
    In addition, we have marked each chamber with the expression realizing \(\dtr^2\). The antipodal chambers
    yield the same type of expression since the inequalities only differ by multiplication with $-1$.
\end{example}

\begin{definition}
  Let \(x,p\in\torus\) and suppose \(x\) is in the complement of the braid arrangement at \(p\). 
  We define the \emph{type of \(x\) with respect to \(p\)} as the unordered pair \(\{\,i,j\,\}\in\binom{[n]}2\)
  such that \[
    \dtr^2(x,p) = (x_i - x_j - (p_i - p_j))^2
  \] and denote this as \(\mathrm{type}(x,p)\).
  For a finite set \(S\) of points in \(\torus\) we define the type of \(x\) with respect to $S$ as \[
    \mathrm{type}(x,S) = \{\,
      \mathrm{type}(x,p) \mid p\in S
    \,\}.
  \]
\end{definition}
Stanley's characterization of the face poset for the braid arrangement \cite{Stanley:1986}
shows that each chamber corresponds to a linear ordering on \([n]= \{1,2,\dots,n\}\). 
In particular, the type is already determined by the choice of signs for \(x_i - x_j\). 

\begin{remark}
    For $n\leq 4$, the types of \(x\) are determined by a coarsening of the braid arrangement as we only need
    the information of which absolute value $\lvert x_i - x_j\rvert$ is maximal. Thus, the type of \(x\)
    depends only on a choice of a minimal element \(j\) and maximal element \(i\) in a partial ordering on \([n]\).

    This is similar to the situation for tropical Fermat--Weber points where \citeauthor{SBYM:2024} observed that
    a refinement of the braid arrangement determines the type of distance function \cite{SBYM:2024}.
    In contrast to their setting, squaring the tropical distances introduces additional symmetries. As a consequence, the objective function \(c\) assumes a specific type on several disjoint convex regions.
\end{remark}

We check that the type is well-defined across the chambers of the braid arrangement, starting with the type of two points with respect to the origin $0\in\torus$.
\begin{lemma}\label{lem:single-braid}
  Let \(x,y,p\in\torus\). If \(x\) and \(y\) are in the same chamber of the braid arrangement at $p$,
  then \(\mathrm{type}(x,p) = \mathrm{type}(y,p)\).

  \begin{proof}
    Since we obtain the braid arrangement at \(p\) by translation, it suffices to check this
    for the braid arrangement \(\braidarr\) (at 0).
  
    The chambers of \(\braidarr\) correspond to linear orders on \([n]\). 
    Thus, \(x\) and \(y\) being in the same chamber mean \[
      x_{\sigma(1)} < x_{\sigma(2)} < \dots < x_{\sigma(n)},\quad\text{and}\quad
      y_{\sigma(1)} < y_{\sigma(2)} < \dots < y_{\sigma(n)}
    \] for some permutation \(\sigma\in\Sym_n\). By forming all possible differences and and comparing 
    accordingly, we get the chain of inequalities \[
      x_{\sigma(n)} - x_{\sigma(1)} 
      \geq x_{\sigma(n)} - x_{\sigma(2)} 
      \geq\dots
      \geq x_{\sigma(n-1)} - x_{\sigma(1)}
      \geq\dots
      \geq x_{\sigma(2)} - x_{\sigma(1)}
      \geq 0.
    \] These inequalities are preserved by squaring, which means that 
    \(\mathrm{type}(x,p) = \mathrm{type}(y,p) = \{\,\sigma(1),\sigma(n)\,\}\).
  \end{proof}
\end{lemma}

\begin{proposition}\label{prop:all-braids}
  Let \(S\subset\torus\) be a finite set and \(x\) and \(y\) be points in the complement of the braid arrangement
  of $S$. Then, \(\mathrm{type}(x,S) = \mathrm{type}(y,S)\) if $x$ and $y$ are in the same chamber of the braid arrangement of $S$.

  \begin{proof}
      Each chamber in the braid arrangement of $S$ arises as the intersection of chambers in the braid arrangements
      at each $p\in S$. Since $x$ and $y$ are in the same chamber of the braid arrangement of $S$,
      we have \(\mathrm{type}(x,p) = \mathrm{type}(y,p)\) for each $p\in S$ by \Cref{lem:single-braid},
      so \(\mathrm{type}(x,S) = \mathrm{type}(y,S)\).
  \end{proof}
\end{proposition}

\begin{figure}
    \centering
\tikzstyle{red dot}=[fill=red, draw=black, shape=rectangle, mark size=2pt]
\tikzstyle{vertex}=[circle, draw=black!85, fill=black!85, thin, scale=0.5]

\definecolor{maxoid_1}{RGB}{51,34,136}
\definecolor{maxoid_2}{RGB}{17,119,51}
\definecolor{maxoid_3}{RGB}{68,170,153}
\definecolor{maxoid_4}{RGB}{136,204,238}
\definecolor{maxoid_5}{RGB}{221,204,119}
\colorlet{maxoid_6}{orange!65}
\definecolor{maxoid_7}{RGB}{204,102,119}
\definecolor{maxoid_8}{RGB}{170,68,153}
\definecolor{maxoid_9}{RGB}{136,34,85}

\tikzstyle{blue_edge}=[-, draw=blue, fill=blue, fill opacity=0.45, line width=1pt]
\tikzstyle{orange_edge}=[-, draw=maxoid_6, fill=maxoid_6, fill opacity=0.2]
\tikzstyle{green_edge}=[-, draw=maxoid_7, fill=maxoid_7, fill opacity=0.2]
\tikzstyle{purple_edge}=[-, draw=maxoid_3, fill=maxoid_3, fill opacity=0.2]
\tikzstyle{red_edge}=[ultra thick, draw=orange, line cap=round]

\pgfdeclarelayer{background}
\pgfdeclarelayer{nodelayer}
\pgfdeclarelayer{edgelayer}
\pgfsetlayers{background, edgelayer, nodelayer}
\begin{tikzpicture}[scale=0.3]
    \begin{pgfonlayer}{background}
        \draw[help lines, color=black, dotted] (-3,-2) grid (13,12);
        \draw[->,ultra thick] (-4,0)--(14,0) node[right]{$x_2-x_1$};
        \draw[->,ultra thick] (0,-3)--(0,13) node[above]{$x_3-x_1$};
    \end{pgfonlayer}
	\begin{pgfonlayer}{nodelayer}
        \node[vertex] (0) at (0, 8) {};
        \node[vertex] (2) at (10, 2) {};
		\node[vertex] (3) at (2, 4) {};
		\node[vertex] (4) at (5, 3) {};
		\node (14) at (4, 4) {};
		\node (28) at (3, 3) {};
	\end{pgfonlayer}
	\begin{pgfonlayer}{edgelayer}
        \draw[style=blue_edge] (0) --++ (4,4);
        \draw[style=blue_edge] (0) --++ (-3,-3);
        \draw[style=blue_edge] (0) --++ (-3,0);
        \draw[style=blue_edge] (0) --++ (13,0);
        \draw[style=blue_edge] (0) --++ (0,4);
        \draw[style=blue_edge] (0) --++ (0,-10);
        
        \draw[style=blue_edge] (2) --++ (-4,-4);
        \draw[style=blue_edge] (2) --++ (3,3);
        \draw[style=blue_edge] (2) --++ (3,0);
        \draw[style=blue_edge] (2) --++ (-13,0);
        \draw[style=blue_edge] (2) --++ (0,-4);
        \draw[style=blue_edge] (2) --++ (0,10);
        
        \draw[style=blue_edge] (3) --++ (-5,-5);
        \draw[style=blue_edge] (3) --++ (8,8);
        \draw[style=blue_edge] (3) --++ (11,0);
        \draw[style=blue_edge] (3) --++ (-5,0);
        \draw[style=blue_edge] (3) --++ (0,-6);
        \draw[style=blue_edge] (3) --++ (0,8);
        
        \draw[style=blue_edge] (4) --++ (-5,-5);
        \draw[style=blue_edge] (4) --++ (8,8);
        \draw[style=blue_edge] (4) --++ (8,0);
        \draw[style=blue_edge] (4) --++ (-8,0);
        \draw[style=blue_edge] (4) --++ (0,-5);
        \draw[style=blue_edge] (4) --++ (0,9);
        
        \draw[style=blue_edge] (3) rectangle (4);
        \draw[style=blue_edge,fill opacity=0.15] (3) rectangle (0);
        \draw[style=blue_edge,fill opacity=0.15] (4) rectangle (2);
        \draw[style=blue_edge,fill opacity=0.15,draw=none] (0) rectangle (-3,12);
        \draw[style=blue_edge,fill opacity=0.15,draw=none] (2) rectangle (13,-2);
        
		\draw[style=red_edge] (14.center) to (28.center);
		\draw[style=red_edge,thick,dotted] (-2,-2) to (12,12);
	\end{pgfonlayer}
\end{tikzpicture}
    \caption{Type decomposition for the four points from \Cref{fig:intersect2}. The shaded chambers all share
             one type of objective function $c$. The corresponding critical locus is the dotted line.}
    \label{fig:covdec}
\end{figure}

\begin{example}\label{ex:pw-braid-fm}
    Consider the four points $S = \{\, (0,0,8),\, (0,2,4),\, (0,5,3),\, (0,10,2) \,\}$ from \Cref{fig:intersect2}.
    The covector decomposition of $S$ is shown in \Cref{fig:covdec}. The tropical Fr\'echet mean lies in the
    deeply shaded chamber. Here, the objective function $c$ is given by the sum of squares \[
        c(x) = (x_3 - x_2 - 8) ^2 + (x_3 - x_2 - 2)^2 + (x_3 - x_2 + 2)^2 + (x_3 - x_2 + 8)^2.
    \] After calculating the intersection of the corresponding chambers and simplifying the facet description,
    we obtain that this chamber is given by \[
        2 \leq x_2 - x_1 \leq 5 \quad\text{and}\quad 3\leq x_3 - x_1 \leq 4.
    \] Since the FM polytrope lies inside this chamber, it can be obtained by calculating the critical
    points of $c$ for this chamber using the partial derivatives \[
        \frac{\partial c}{\partial x_1} = 0,\quad
        \frac{\partial c}{\partial x_2} = 8x_2 - 8x_3,\quad\text{and}\quad
        \frac{\partial c}{\partial x_2} = 8x_3 - 8x_2.
    \] From this, we can see that the critical locus of $c$ is precisely the line given by $x_3 - x_2$
    whose intersection with the dark shaded chamber gives the FM polytrope.

    There are additional chambers which share the same type of objective function (lightly shaded), but they
    do not intersect the critical locus.
\end{example}

Having the critical locus inside of the corresponding chamber is the straightforward case, since we
only need to calculate the intersection of an affine subspace with a polyhedron. It is then immediate that
this intersection yields the FM polytrope.

If the critical locus lies outside of the corresponding chamber, the FM polytrope arises as the boundary optimum
of this chamber. As a word of caution, this boundary optimum cannot be obtained as the nearest point projection
onto the chamber, as we demonstrate in the next example.

\begin{figure}
    \centering
\tikzstyle{red dot}=[fill=red, draw=black, shape=rectangle, mark size=2pt]
\tikzstyle{vertex}=[circle, draw=black!85, fill=black!85, thin, scale=0.5]

\definecolor{maxoid_1}{RGB}{51,34,136}
\definecolor{maxoid_2}{RGB}{17,119,51}
\definecolor{maxoid_3}{RGB}{68,170,153}
\definecolor{maxoid_4}{RGB}{136,204,238}
\definecolor{maxoid_5}{RGB}{221,204,119}
\colorlet{maxoid_6}{orange!65}
\definecolor{maxoid_7}{RGB}{204,102,119}
\definecolor{maxoid_8}{RGB}{170,68,153}
\definecolor{maxoid_9}{RGB}{136,34,85}

\tikzstyle{blue_edge}=[-, draw=blue, fill=blue, fill opacity=0.45, line width=1pt]
\tikzstyle{orange_edge}=[-, draw=maxoid_6, fill=maxoid_6, fill opacity=0.2]
\tikzstyle{green_edge}=[-, draw=maxoid_7, fill=maxoid_7, fill opacity=0.2]
\tikzstyle{purple_edge}=[-, draw=maxoid_3, fill=maxoid_3, fill opacity=0.2]
\tikzstyle{red_edge}=[ultra thick, draw=orange, line cap=round]

\pgfdeclarelayer{background}
\pgfdeclarelayer{nodelayer}
\pgfdeclarelayer{edgelayer}
\pgfsetlayers{background, edgelayer, nodelayer}
\begin{tikzpicture}[scale=0.3]
    \begin{pgfonlayer}{background}
        \draw[help lines, color=black, dotted] (-11,-14) grid (14,14);
        \draw[->,ultra thick] (-12,0)--(15,0) node[right]{$x_2-x_1$};
        \draw[->,ultra thick] (0,-15)--(0,15) node[above]{$x_3-x_1$};
    \end{pgfonlayer}
	\begin{pgfonlayer}{nodelayer}

        \node[vertex,orange] (tfm) at (0,-1) {};
        \node[vertex,maxoid_3] (candidate_1) at (6,3) {};
        \node[vertex,maxoid_8] (candidate_1) at (-10.5,-4.5) {};
        
        \node[vertex] (0) at (3,3) {};
        \node[vertex] (2) at (-6,0) {};
		\node[vertex] (3) at (0,-12) {};
		\node[] (4) at (5, 3) {};
		\node (14) at (4, 4) {};
		\node (28) at (3, 3) {};
	\end{pgfonlayer}
	\begin{pgfonlayer}{edgelayer}
        \draw[style=blue_edge,maxoid_8] (0,0) --++ (3,0) -- ($(3)+(3,3)$) -- (3.center) -- cycle;
        \draw[style=blue_edge,maxoid_3,line cap=round] (0,0) --++ (-6,-6) -- ($(3)-(6,0)$) -- (3.center) -- cycle;

        \draw[style=blue_edge] (0) --++ (11,11);
        \draw[style=blue_edge] (0) --++ (-14,-14);
        \draw[style=blue_edge] (0) --++ (-14,0);
        \draw[style=blue_edge] (0) --++ (11,0);
        \draw[style=blue_edge] (0) --++ (0,11);
        \draw[style=blue_edge] (0) --++ (0,-17);
        
        \draw[style=blue_edge] (2) --++ (-5,-5);
        \draw[style=blue_edge] (2) --++ (14,14);
        \draw[style=blue_edge] (2) --++ (20,0);
        \draw[style=blue_edge] (2) --++ (-5,0);
        \draw[style=blue_edge] (2) --++ (0,-14);
        \draw[style=blue_edge] (2) --++ (0,14);
        
        \draw[style=blue_edge] (3) --++ (-2,-2);
        \draw[style=blue_edge] (3) --++ (14,14);
        \draw[style=blue_edge] (3) --++ (14,0);
        \draw[style=blue_edge] (3) --++ (-11,0);
        \draw[style=blue_edge] (3) --++ (0,-2);
        \draw[style=blue_edge] (3) --++ (0,26);
        
        
        
	\end{pgfonlayer}
\end{tikzpicture}
    \caption{Type decomposition for the three points from \Cref{fig:intersect}.}
    \label{fig:covdec2}
\end{figure}

\begin{example}
    The three points $P = \{\, (-3,0,0),\, (0,-6,0),\, (0,0,-12) \,\}$ from \Cref{ex:quadratic-optimisation}
    have the unique tropical Fr\'echet mean $\tfm = (0,0,-1)$. \Cref{fig:covdec2} shows that
    \(\tfm\) lies on the boundary of two chambers. We focus on the computations for the left chamber as 
    the computations for the right chamber are carried out analogously.

    On the right chamber, the objective function is given by the expression \[
        c(x) = (x_3 - x_1 - 3)^2 + (x_3 - x_2 - 6)^2 + (x_3 - x_1 + 12)^2
    \] with partial derivatives \[
        \frac{\partial c}{\partial x_1} =  4x_1        - 4x_3 - 18,\quad
        \frac{\partial c}{\partial x_2} =         2x_2 - 2x_3 + 12,\quad\text{and}\quad
        \frac{\partial c}{\partial x_3} = -4x_1 - 2x_2 + 6x_3 +  6.
    \] The linear system of equations given by the vanishing of these partial derivatives has the unique solution
    \(x' = (0,-\frac{21}{2},-\frac92)\). This lies outside the chamber which is given by the linear inequalities \[
        0 \leq x_2 - x_1 \leq 3,\quad
               x_3 - x_1 \leq 0
        \quad\text{and}\quad
      -12 \leq x_3 - x_2 .
    \] Notably, the actual tropical Fr\'echet mean $\tfm$ lies on the boundary of this chamber,
    but it is not the unique nearest point to the critical point.
    For example, $(0,0,-\frac92)$ is also contained in this chamber and \[
        \dtr^2\left(\left(0,0,-\frac92\right),\left(0,-\frac{21}2,-\frac92\right)\right) = \left(\frac{21}{2}\right)^2 = \dtr^2(\tfm,(0,6,3)).
    \] To obtain the boundary optimum of $c$ on this chamber, we can consider the Lagrangian \[
        \mathcal{L}(x,\lambda) = c(x) + \lambda\cdot(x_2 - x_1)
    \] to find the critical points of $c$ along the hyperplane $x_2 - x_1$. In this case, we
    consider the partial derivatives \[
        \frac{\partial\mathcal{L}}{\partial x_1} = \frac{\partial c}{\partial x_1} - \lambda,\ 
        \frac{\partial\mathcal{L}}{\partial x_2} = \frac{\partial c}{\partial x_2} + \lambda,\ 
        \frac{\partial\mathcal{L}}{\partial x_3} = \frac{\partial c}{\partial x_3},\ \text{and}\ 
        \frac{\partial\mathcal{L}}{\partial \lambda} = x_2 - x_1. 
    \] These are still linear equations which now have the solution $(x,\lambda) = (0,0,-1,-14)$,
    meaning that we have successfully recovered the tropical Fr\'echet mean.
\end{example}

In fact, the critical loci are always intersections of affine hyperplanes since we consider sums of squares.
This suggests the following procedure to calculate the FM polytrope exactly.

For the following discussion, we first fix some notation. We denote by $A_n\in\R^{N\times n}$ the matrix whose rows are differences $e_i - e_j$
of the standard basis vectors for $1\leq j<i\leq n$. Thus, $N = \binom{n}2$.
By abuse of notation, we call $s\in\{\,\pm\,\}^{m\times N}$ a covector for the point set $P$.
We denote by $\mathcal{C}(s)$ the chamber corresponding to a covector $s$.
A point \(x\) from the interior of the corresponding chamber gives rise to a type $\mathrm{type}(x,P)$.

\begin{algorithm}
  \caption{Computing the FM polytrope using the braid arrangement}
    \label{alg:frechet-braid}
    
    \begin{algorithmic}[1]
    \Procedure{BraidChambersFr\'echetMean}{}
    \State \textbf{Input:} data $P = \{p^{(1)}, \dots, p^{(m)}\} \subset \torus$
    \State \textbf{Output:} Fréchet mean polytope $\overline{P}$ 
    \For{every covector \(s\in\{\pm\}^{m\times N}\) of \(\braidarr(P)\)} \Comment{$2^{m\cdot N}$ iterations}
    \State Compute $\mathrm{type}(x,P)$ for generic $x$ with covector $s$ 
    \Comment{$\bigO(mn^3)$ by \Cref{lem:covector-to-type-algo}}
    \State Compute the critical locus of $c$ given $s$
    \Comment{$\bigO(n^3)$}
    \If {critical locus intersects \(\mathcal{C}(s)\)} \Comment{$\bigO(2^n)$}
    \State \(\overline{P}_s\leftarrow\) intersection of the critical locus with \(\mathcal{C}(s)\)
    \Else
    \State \(\overline{P}_s\leftarrow\) locus of the boundary optimum for \(\mathcal{C}(s)\)
    \Comment{$\bigO(n^5)$ by \Cref{lem:polytrope-lagrangian}}
    \EndIf
    \EndFor
    \State \textbf{Return} $\overline{P}_s$ maximizing $c$
    \EndProcedure
    \end{algorithmic}
\end{algorithm}

In the worst case, steps 5-8 are repeated $2^{m\cdot N}$ times which can be improved by only considering feasible covectors
and considering the symmetry in the sign choices. 
It is possible to take only covectors with positive first entry, reducing the number of iterations by a factor 
of $2^N$. Since one would then need to consider all chambers corresponding to an orbit in step $7$,
this does not improve the overall complexity of \Cref{alg:frechet-braid}

\begin{lemma}\label{lem:covector-to-type-algo}
    Let \(P\subset\torus\) be a set of points and $x$ be a point with covector $s\in\{\pm\}^{m\times N}$.
    Given $s$, \(\mathrm{type}(x,P)\) can be calculated in $\bigO(mn^3)$ steps.

    \begin{proof}
      Computing the $\mathrm{type}(x,P)$ means that we need to compute $\mathrm{type}(x,p^{(\nu)})$ for each $p^{(\nu)}\in P$.
      We can form a matrix $A_n\in\{\,-1,0,1\,\}^{N\times n}$ whose rows are the differences of the basis vectors
      \(e_i - e_j\) for \(1\leq j<i\leq n\). If we denote by $s^{(\nu)}$ the covector of $x$ with respect to
      \(\braidarr(p^{(\nu)})\), we can multiply each row of $A_n$ by the corresponding sign $s_{ij}^{(\nu)}$.
      Analogously to the proof of \Cref{lem:single-braid}, there will be a column with only positive entries
      corresponding to $\sigma(n)$ and a column with only negative elements corresponding to $\sigma(1)$.

      Building up $A_n$ takes $\bigO(n^2)$ steps since $N = \binom{n}2$, and applying each covector 
      $s^{(\nu)}$ to $A_n$ takes $\bigO(n^3)$ steps. Likewise, we can find \(\sigma(n)\) and \(\sigma(1)\)
      by scanning each column once which again takes $\bigO(n^3)$ in total. Since we have to do this for
      all $m$ covectors $s^{(\nu)}$ in $s$ we need $\bigO(mn^3)$ steps to obtain \(\mathrm{type}(x,P)\).
    \end{proof}
\end{lemma}

To compute the critical locus of \(c\) of a given type \(\mathrm{type}(x,P)\), we only need to calculate 
the gradient of a quadratic polynomial in $n$ variables. These will in particular be linear polynomials, 
thus the critical locus is an affine subspace. Thus, we can obtain the critical locus by solving 
an affine-linear system of equations.
If this subspace intersects the corresponding chamber, we already have a candidate for the FM polytrope
assuming that the minimum of $c$ is indeed realized by the corresponding type.

If the critical locus \emph{does not} intersect the chamber then the actual critical points of $c$ over this
chamber lie on the boundary. To obtain these boundary critical points, we can use the method 
of Lagrange multipliers for the facets of the chamber which in our case introduce only additional linear equations.

Without any heuristic, we must set up a Lagrangian for every facet of a particular chamber.
Since each chamber in the braid arrangement is a polytrope, we know that there are at most \(\binom{n}{2}\) facets.
As the Lagrangian is of the form $\mathcal{L} = c+\lambda\cdot g$ where $g$ defines the hyperplane
supporting a facet, the critical locus is still given by a affine-linear system of equations.

\begin{lemma}\label{lem:polytrope-lagrangian}
    If $f\in\R[x_1,\dots,x_n]$ is the sum of squares of linear polynomials and $P\subset\torus$ is
    a polytrope given by inequalities, the minimum of $f$ over $P$ can be found in $\bigO(2^n)$ steps.

    \begin{proof}
        To compute the critical locus $Q$ of \(f\), we only need to calculate the gradient of a 
        quadratic polynomial in $n$ variables. These will in particular be linear polynomials, 
        thus the critical locus is an affine subspace. Thus, we can obtain the critical locus by solving 
        an affine-linear system of equations. This can be done in $\bigO(n^3)$ steps when using Gaussian elimination.

        The intersection of $P$ and $Q$ can be found using the simplex algorithm in exponential time 
        \cite{MR332165}.
        If the critical locus of $f$ lies outside $P$, we have to compute a boundary minimum for $f$
        instead. This can be done using the Lagrangian \[
          \mathcal{L} = c+\lambda\cdot g
        \] where $g$ defines the hyperplane supporting a facet. The critical locus of the Lagrangian is 
        still given by affine-linear equations. In the worst case, we have to solve such a system for every 
        facet of $P$, of which there are at most \(\binom{n}{2}\).
        This means that we have to carry out $\bigO(n^5)$ steps in the worst case.
    \end{proof}
\end{lemma}

Combining the complexity estimates for each step of \Cref{alg:frechet-braid}, we obtain 
the following result.
\begin{theorem}
  For \(S = \{\, p^{(1)}, \dots, p^{(m)} \,\}\subset\torus\), it takes \(\bigO(2^{mn^2+n}\cdot(mn^3))\) steps to compute the FM polytrope exactly.
\end{theorem}

\section{Numerical Experiments}

In this section, we discuss the implementation to compute tropical Fréchet means and Fréchet mean polytropes, and present results of numerical experiments.

This study deals with two computational problems: calculating a single Fréchet mean and calculating 
the FM polytrope from given data. The computation of the latter depends on the former, in other words, to calculate the FM polytrope, we require a single tropical Fréchet mean.  The single tropical Fréchet mean may be computed using \Cref{alg:greedy-frechet}, from which the FM polytrope may then be computed 
using the procedure outlined in \Cref{prop:fm-polytrope-compute}.  

The computation of the FM polytrope is carried out using exact arithmetic and its accuracy depends on that of the initial Fréchet mean. Our experiments focus on testing this connection and better understanding the robustness of the FM polytrope computation subject to various Fréchet mean initiations.

We employ the solvers \texttt{Clarabel.jl} \cite{Clarabel_2024} and \texttt{Ipopt} \cite{Ipopt},
for which there exists a Julia binding supporting \texttt{MathOptInterface}.

\begin{example}
    Consider the three points \begin{small}\[
        P = \left\{\,
            \left(\frac{1}{5}, \frac{2}{5}, 2, \frac{2}{5},2,2 \right),\, 
            \left(2, 2, 2, \frac{2}{5}, \frac{2}{5}, \frac{1}{5} \right) ,\,
            \left(\frac{2}{5}, \frac{2}{5}, 2, \frac{1}{5},2,2 \right) 
        \,\right\}.
    \]\end{small}

    Calculating a tropical Fréchet mean numerically with \texttt{Clarabel} gives the following candidate \[
        x_1 \approx (-0.773, -0.773, 0.034, -1.048, -0.172, 0.034),
    \] which attains a sum of squared tropical distances of \[
        c(x_1) \approx 7.2800000000839\dots > 7.28.
    \]
    Calculating the FM polytrope \(\overline{P}\) improves on this and instead gives points which exactly 
    attain the sum of squared tropical distances \(\frac{182}{25} = 7.28\).

    We can moreover certify that the vertices of \(\overline{P}\) are indeed tropical Fréchet means using the approach
    from \Cref{thm:conv}. The corresponding lower bounds on the tropical distances are given by 
    \[
        \dtr^{2}(x,p_{1})\ge \left(x_1 - x_3 + \frac{9}{5}\right)^2,\ 
        \dtr^{2}(x,p_{2})\ge \left(x_2 - x_5 - \frac{8}{5}\right)^2,\ 
        \dtr^{2}(x,p_{3})\ge \left(x_4 - x_6 - \frac{9}{5}\right)^2.
    \]

    Checking for \(\bar{x} = \left( -\frac{3}{5}, -\frac35, 0, -\frac45, 0, 0\right)\) gives the following lower bound 
    for the sum of squares \[
    	\begin{split}
    		\sum_{i=1}^{3}{\dtr^{2}(x,p_{i})}
                &\ge \left(x_1 - x_3 + \frac{9}{5}\right)^2 + \left(x_2 - x_5 - \frac{8}{5}\right)^2 
                    + \left(x_4 - x_6 - \frac{9}{5}\right)^2 \\
                &= x_1^2 - 2 x_1x_3 + \frac{18}{5} x_1 + x_2^2 - 2 x_2x_5 - \frac{16}{5} x_2 + x_3^2 
                    -\frac{18}{5} x_3 + x_4^2 - 2 x_4x_6 \\
                &\qquad+ \frac{18}{5} x_4 + x_5^2 + \frac{16}{5} x_5 + x_6^2  -\frac{18}{5} x_6 + \frac{226}{25} \\
                &= \left(x_1 - x_3 + \frac35\right)^2 + \left(x_2 - x_3 + \frac35\right)^2
                    +\left(x_4 + x_3 + \frac45\right)^2 \\
                &\qquad + 2\left(x_4 - x_5 + \frac85\right)^2 + \frac{182}{25} \ge \frac{182}{25}. \qedhere
    	\end{split}
	\] 
\end{example}

We computed several examples like the above by hand and then employed the aforementioned solvers to obtain a candidate for a tropical Fréchet mean and 
compare it to the manually calculated solution. We compared the performance of optimization
according to \Cref{sec:exact-quadratic} to the numerical computation using \Cref{alg:greedy-frechet}.
In our experiments we found that in general the numerically obtained Fréchet mean is very close to the 
actual tropical Fréchet mean, up to an error of order $10^{-5}$.

We also ran experiments for randomly generated data. We generated 10 samples each for 
dimensions \(n\in\{5,10,15,20\}\) and sample sizes \(m\in\{n, 2n, 3n\}\).
The experiments were run on a MacBook Pro with an 2,4 GHz Apple M2 Max. Each computation was allotted a maximum of
\num{16}\unit{\giga\byte} of memory.
The average runtime for each dimension and sample size is shown in \Cref{fig:timings}.

\begin{figure}[b]
    \centering
    \begin{tikzpicture}
\begin{axis}[
ybar,
ymin=0,
ylabel={Average time elapsed (in seconds)},
xlabel={Dimension},
bar width=7pt,
symbolic x coords={5,10,15,20},
xtick=data,
xtick pos=left,
ytick pos=left,
legend style={at={(0.5,-0.25)},
anchor=north,legend columns=-1},
]
    \addplot
        coordinates {(5,0.0592454) (10,0.26136) (15,0.831676) (20,1.99134)};
    \addplot
        coordinates {(5,0.0546671) (10,0.369878) (15,1.17037) (20,3.1957)};
    \addplot
        coordinates {(5,0.0731361) (10,0.428845) (15,1.65809) (20,4.13591)};
    \legend{$n$, $2n$, $3n$}
\end{axis}
\end{tikzpicture}
    \caption{Average timings of calculating tropical Fréchet means for random samples
    in dimension \(n\in\{5,10,15,20\}\) and sample sizes \(m\in\{n, 2n, 3n\}\).}
    \label{fig:timings}
\end{figure}

\subsection*{Software Availability}

Our implementation in \texttt{Julia} for computing tropical Fré\-chet means is freely available at 
\begin{center}
\url{https://zenodo.org/records/14811067}
\end{center}


\section{Discussion}
\label{sec:discussion}

In this paper, we studied the optimization problem of computing Fréchet means under the tropical metric in the tropical projective torus.  Tropical Fréchet means are in general not unique; we give a characterization of the set of all tropical Fréchet means as a classical and tropical convex polytope, and prove existence of a nonnegative certificate for the computation of tropical Fréchet means.  We studied the exact computation of tropical Fréchet means and the tropical Fréchet means polytrope; we proposed and implemented an algorithm to compute tropical Fréchet means and tested it numerically.  We now discuss several future research directions.

We have seen that the behavior of the FM polytrope for a given sample is widely varying, ranging from
unique Fréchet means to polytropes of different dimensions. As the intersection of boundaries of
tropical balls, it is immediate that the FM polytrope must have codimension at least one.
Beyond this observation, it is not clear how to predict the dimension of the FM polytrope.
Thus, criteria for the expected dimension of an FM polytrope, or even uniqueness of tropical Fréchet means is an interesting question for further study.

With a view towards statistical applications, it is also reasonable to consider tropical Fréchet means
up to an error or \(\varepsilon\). In this setting, the constants provided by \Cref{thm:frechet} become intervals,
and the FM polytrope is obtained as the intersection of \(\varepsilon\)-thickenings of tropical balls.
\citeauthor{JS:2022} \cite{JS:2022} studied a parameterized version of polytopes and the question of finding
the feasible set of weights for a given parameterization such that the resulting polytrope is nonempty.  Our work serves as a possible starting point for computing a FM polytrope up to an error and computing tropical Fréchet means from real data.

The classical definition of Fr\'echet means involves the minimization of a sum of squared distances.
This may also be generalized to sums of arbitrary positive powers---that is, finding points minimizing the 
objective function \[
    \sum_{\nu=1}^{m}{\dtr^q(\tfm,p^{(\nu)})},
\] leading to a notion of \emph{tropical \(q\)-Fr\'echet means}. For \(q=1\), this recovers the tropical Fermat--Weber problem
as has been observed by \citeauthor{Comǎneci:2024}~\cite{Comǎneci:2024} and previously studied by \citeauthor{Lin.Yoshida:2018} \cite{Lin.Yoshida:2018}. A natural question then follows on the properties
of tropical \(q\)-Fr\'echet means for \(q\neq 1,2\) and how existing results on tropical Fermat--Weber points
and those established in this paper on tropical Fréchet means may be generalized to the arbitrary $q$ setting. 
The cases \(q=1,2\) constitute tropically convex location problems in the sense of \citeauthor{Comǎneci:2024} \cite{Comǎneci:2024},
which then opens the question of whether the remaining cases also exhibit tropical convexity in the same sense and whether they may be studied in the same framework.

We have also presented a local method that breaks down the optimization problem for tropical Fr\'echet means into a large number
of small linear problems. While combinatorially demanding, this method is amenable for computer algebra systems which are capable of
producing exact solutions. This also serves as a starting point of a connection between tropical geometry
and \emph{metric algebraic geometry} \cite{BKS:2024}. A possible direction for future work involves the study of conditions on the sample
points such that the FM polytrope is positive dimensional.
As suggested by \Cref{ex:pw-braid-fm}, this problem is connected to gradients of pieces of the objective function
vanishing identically zero.

In summary, this paper shows that the tropical Fréchet mean problem admits a fully symbolic and polyhedral formulation, in which both the geometry of the solution set and the structure of the objective function can be described combinatorially.  Our work places tropical Fréchet means within the broader context of exact quadratic optimization problems on polyhedral spaces and suggests a general paradigm for treating metric optimization problems in tropical and other piecewise-linear geometries using symbolic methods.

\section*{Acknowledgments}
We would like to extend special thanks to Ariff Jazlan Johan for contributing to \Cref{alg:greedy-frechet} and helping with implementations.  
We thank Marco Botte, Stephan Huckemann, Beatrice Matteo, and Roan Talbut for helpful discussions.
We also thank anonymous referees for their positive feedback and comments.

We would like to thank the Centro de Investigaci\'on en Matem\'aticas and Adri\'an P\'erez Valencia
for their hospitality during the International Symposium on Symbolic and Algebraic Computation 2025 which took
place in Guanajuato, M\'exico.

C.\,A.\ and K.\,F.\ acknowledge funding by the Deutsche Forschungsgemeinschaft (DFG, German Research
Foundation) under Germany´s Excellence Strategy – The Berlin Mathematics
Research Center MATH+ (EXC-2046/1, EXC-2046/2, project ID: 390685689). A.\,M.\ is supported by the UK Research and Innovation: Engineering and Physical Sciences Research Council under grant reference [EP/Y028872/1].

\printbibliography


\end{document}